\documentclass[a4paper,reqno,11pt]{amsart}
\usepackage{amsmath,amsthm,mathtools}
\usepackage[hidelinks]{hyperref}
\usepackage{amsfonts}
\usepackage{amssymb}
\usepackage{color}
\usepackage{fullpage}
\usepackage{bbm}
\usepackage{graphics}
\usepackage{tikz}
\usetikzlibrary{shapes.geometric}
\usetikzlibrary{decorations.pathreplacing}
\usetikzlibrary{cd}

\makeatletter
\@namedef{subjclassname@2020}{\textup{2020} Mathematics Subject Classification}
\makeatother

\newtheorem{theorem}{Theorem}[section]
 
 \newtheorem{proposition}[theorem]{Proposition}
 
 \newtheorem{lemma}[theorem]{Lemma}

 \theoremstyle{definition}

 \newtheorem{definition}[theorem]{Definition}

\newcommand{\GETOUT}[1]{}
\newcommand{\toprob}[1]{\stackrel{#1}{\mapsto}}

\newcommand{\ds}{\displaystyle}

\renewcommand{\P}{\mathbb{P}}
\newcommand{\E}{\mathbb{E}}

\newcommand{\Fin}{\mathsf{Manila}}
\newcommand{\Manila}{\Fin}
\newcommand{\manila}{\mathsf{manila}}
\newcommand{\Paths}{\mathsf{Paths}}
\newcommand{\paths}{\mathsf{paths}}
\newcommand{\pdz}{\frac{d}{dz}}
\newcommand{\ppdz}{\frac{d^2}{dz^2}}
\newcommand{\Raney}{\mathsf{Raney}}
\newcommand{\Mean}{\mathsf{First}}

\makeatletter
\def\specialsection{\@startsection{section}{1}%
  \z@{\linespacing\@plus\linespacing}{.5\linespacing}%
  {\normalfont}}
\def\section{\@startsection{section}{1}%
  \z@{.7\linespacing\@plus\linespacing}{.5\linespacing}%
  {\noindent\normalfont\bf}}
\makeatother

\newcommand{\finfigone}{
\begin{figure}[!h]
\begin{center}
\begin{tikzpicture}
\begin{scope}[rotate=90]
\node at (1,1.7) {(a)};
\draw[thick] (0,0) .. controls ++(0.6,0.4) .. ++(1,1);
\draw[thick] (0,0) -- ++(1.2,0);
\draw[thick] (0,0) .. controls ++(0.6,-0.4) .. ++(1,-1);
\draw[thick] (0,-2.5) .. controls ++(0.6,0.4) .. ++(1,1);
\draw[thick] (0,-2.5) -- ++(1.2,0);
\draw[thick] (0,-2.5) .. controls ++(0.6,-0.4) .. ++(1,-1);
\node [rotate=90] at (0.5,-3.75) {\tiny{front}};
\node [rotate=90] at (0.5,1.25) {\tiny{back}};
\end{scope}
\begin{scope}[xshift=5.5cm,rotate=90]
\node at (1,0.5) {(b)};
\draw[thick] (0,-1.250) .. controls ++(0.6,0.4) .. ++(1,1);
\draw[thick] (0,-1.250) -- ++(1.2,0);
\draw[thick] (0,-1.250) .. controls ++(0.6,-0.4) .. ++(1,-1);
\draw[thick] (0.4,-1.25+0.15) .. controls ++(0.5,0.4) .. ++(0.7,0.7);
\draw[thick] (0.4,-1.25+0.15) .. controls ++(0.6,0.1) .. ++(0.9,0.1);
\draw[thick] (0.4,-1.25+0.15) -- ++(0.8,0.4);
\end{scope}
\begin{scope}[xshift=10cm,rotate=90]
\node at (1,0.5) {(c)};
\draw[thick] (0,-1.250) .. controls ++(0.6,0.4) .. ++(1,1);
\draw[thick] (0,-1.250) -- ++(1.2,0);
\draw[thick] (0,-1.250) .. controls ++(0.6,-0.4) .. ++(1,-1);
\draw[thick] (0.4,-1.25-0.15) .. controls ++(0.5,-0.4) .. ++(0.7,-0.7);
\draw[thick] (0.4,-1.25-0.15) .. controls ++(0.6,-0.1) .. ++(0.9,-0.1);
\draw[thick] (0.4,-1.25-0.15) -- ++(0.8,-0.4);
\end{scope}
\end{tikzpicture}
\end{center}
\caption{There are three ways to arrange two Manlia folders, each with two compartments. $|\Fin_3(2,0)|=2$ and $|\Fin_3(2,1)|=1$.\label{fig:manila}}
\end{figure}
}

\newcommand{\finfigtwo}{
\begin{figure}[!h]
\begin{center}
\begin{tikzpicture}
\begin{scope}[rotate=90]
\node at (1,1.7) {(a)};
\draw[thick] (0,0) .. controls ++(0.6,0.4) .. ++(1,1);
\node[align=right] at (1.15,1) {\tiny $L$};
\draw[thick] (0,0) -- ++(1.2,0);
\node[align=right] at (1.35,0) {\tiny $L$};
\draw[thick] (0,0) .. controls ++(0.6,-0.4) .. ++(1,-1);
\node[align=right] at (1.15,-1) {\tiny $D$};
\draw[thick] (0,-2.5) .. controls ++(0.6,0.4) .. ++(1,1);
\node[align=right] at (1.15,-1.5) {\tiny $L$};
\draw[thick] (0,-2.5) -- ++(1.2,0);
\node[align=right] at (1.35,-2.5) {\tiny $L$};
\draw[thick] (0,-2.5) .. controls ++(0.6,-0.4) .. ++(1,-1);
\node[align=right] at (1.15,-3.5) {\tiny $D$};
\node [rotate=90] at (0.5,-3.75) {\tiny{front}};
\node [rotate=90] at (0.5,1.25) {\tiny{back}};
\end{scope}
\begin{scope}[xshift=5.5cm,rotate=90]
\node at (1,0.5) {(b)};
\draw[thick] (0,-1.250) .. controls ++(0.6,0.4) .. ++(1,1);
\draw[thick] (0,-1.250) -- ++(1.2,0);
\draw[thick] (0,-1.250) .. controls ++(0.6,-0.4) .. ++(1,-1);
\node[align=right] at (1.15,-1.25+1) {\tiny $L$};
\node[align=right] at (1.35,-1.25+0) {\tiny $L$};
\node[align=right] at (1.15,-1.25+-1) {\tiny $D$};
\draw[thick] (0.4,-1.25+0.15) .. controls ++(0.5,0.4) .. ++(0.7,0.7);
\node[align=right] at (1.30,-1.25+0.15+0.7) {\tiny $L$};
\draw[thick] (0.4,-1.25+0.15) .. controls ++(0.6,0.1) .. ++(0.9,0.1);
\node[align=right] at (1.50,-1.25+0.15+0.1) {\tiny $D$};
\draw[thick] (0.4,-1.25+0.15) -- ++(0.8,0.4);
\node[align=right] at (1.40,-1.25+0.15+0.4) {\tiny $L$};
\end{scope}
\begin{scope}[xshift=10cm,rotate=90]
\node at (1,0.5) {(c)};
\draw[thick] (0,-1.250) .. controls ++(0.6,0.4) .. ++(1,1);
\draw[thick] (0,-1.250) -- ++(1.2,0);
\draw[thick] (0,-1.250) .. controls ++(0.6,-0.4) .. ++(1,-1);
\node[align=right] at (1.15,-1.25+1) {\tiny $L$};
\node[align=right] at (1.35,-1.25+0) {\tiny $L$};
\node[align=right] at (1.15,-1.25-1.05) {\tiny $D$};
\draw[thick] (0.4,-1.25-0.15) .. controls ++(0.5,-0.4) .. ++(0.7,-0.7);
\node[align=right] at (1.30,-1.25-0.15-0.7+0.01) {\tiny $D$};
\draw[thick] (0.4,-1.25-0.15) .. controls ++(0.6,-0.1) .. ++(0.9,-0.1);
\node[align=right] at (1.50,-1.25-0.15-0.1+0.01) {\tiny $L$};
\draw[thick] (0.4,-1.25-0.15) -- ++(0.8,-0.4);
\node[align=right] at (1.40,-1.25-0.15-0.4) {\tiny $L$};
\end{scope}
\begin{scope}[xshift=0cm,yshift=-1.75cm,scale=0.6]
\draw[gray,thin] (0,0) grid (4,2);
\draw[very thick] (4,2) -- (4,1) -- (2,1) -- (2,0) -- (0,0);
\draw[gray,thin,dashed] (0,0) -- (4,2);
\foreach \x/\y in {4/2,4/1,3/1,2/1,2/0,1/0,0/0}{
	\draw[fill=black] (\x,\y) circle (0.5ex);
	}
\end{scope}
\begin{scope}[xshift=5.5cm,yshift=-1.75cm,scale=0.6]
\draw[gray,thin] (0,0) grid (4,2);
\draw[very thick] (4,2) -- (4,1) -- (3,1) -- (3,0) -- (0,0);
\draw[gray,thin,dashed] (0,0) -- (4,2);
\foreach \x/\y in {4/2,4/1,3/1,3/0,2/0,1/0,0/0}{
	\draw[fill=black] (\x,\y) circle (0.5ex);
	}
\end{scope}
\begin{scope}[xshift=10cm,yshift=-1.75cm,scale=0.6]
\draw[gray,thin] (0,0) grid (4,2);
\draw[very thick] (4,2) -- (4,0) -- (0,0);
\draw[gray,thin,dashed] (0,0) -- (4,2);
\foreach \x/\y in {4/2,4/1,4/0,3/0,2/0,1/0,0/0}{
	\draw[fill=black] (\x,\y) circle (0.5ex);
	}
\end{scope}
\end{tikzpicture}
\end{center}
\caption{
Illustration of the bijection between arrangements of Manila folders and paths through the lattice.\label{fig:bij:illustation}}
\end{figure}
}

\title{Knuth's big-chooser matchbox process: the case of many matchboxes}
\author{Mark Dukes and Andrew Mullins}
\address{School of Mathematics and Statistics, University College Dublin, Dublin 4, Ireland.}
\email{mark.dukes@ucd.ie}
\email{andrew.mullins@ucdconnect.ie}

\begin{document}
\begingroup
\def\uppercasenonmath#1{} 
\let\MakeUppercase\relax 
\maketitle
\endgroup

\begin{abstract}
Banach's matchbox problem considers the setting of two matchboxes that each initially contain the same number of matches.
Boxes are chosen with equal probability and a match removed each time. 
The problem concerns the law of the number of matches remaining in one box once the other box empties.
Knuth considered a generalization of this problem whereby \textit{big-choosers} arrive with probability $p$ and 
remove a match from the box with the most number remaining, and \textit{little-choosers} arrive with probability $1-p$ and remove a match from the box with the least number remaining.

In this paper we consider Knuth's generalization for the case of $k$ matchboxes in which there are \textit{big-choosers} and \textit{little-choosers}. 
We determine the generating function for the expected number of matches remaining in $k-1$ matchboxes once a box first empties, a quantity we refer to as the `residue'.
Interestingly, this generating function is a quotient whose denominator contains a generating function for a special case of the Raney numbers. 
The form for this generating function allows us to give an expression for the expected residue in terms of a sum that involves diagonal state return probabilities, 
where a diagonal state is a configuration in which all matchboxes each contain the same number of matches.
We use analytic techniques to determine the asymptotic behaviour of this expected value for all values of $p$, which involves the study of an asymmetric random walk.

In addition to this we consider the expected value of the order of the first return to a diagonal state and determine the asymptotic behaviour of this quantity.
The coefficients of the diagonal state probability generating function are shown to be related to `manila folder configurations in a filing cabinet', and we make this connection precise. 
This allows us to use known results for the enumeration of such manila folder configurations to give a closed form expression for the diagonal state return probabilities.
\end{abstract}

\tableofcontents

\section{Introduction}\label{background}
Banach's matchbox problem~\cite[Sec. 8]{feller} considered two matchboxes each initially containing $n$ matches.
At every time-step a box is selected at random and a match removed. 
The problem is to determine the probability that one box contains a prescribed number of matches at the moment the other box empties. 
In his playfully entitled `The Toilet Paper Problem' paper, Knuth~\cite{knuth} generalized this problem by introducing a choice-rule relating to the box\footnote{The setting of Knuth's generalization considered two rolls of toilet paper, each initially containing $n$ sheets, instead of matchboxes. We choose to use the matchbox terminology.} 
containing the most number of matches.
In his generalization, at every time-step the box with the most number of matches remaining is selected with probability $p$ and a match removed, while the other box is selected with probability $q=1-p$. 
He showed that the expected number of matches remaining in one box once the other box empties (a quantity he refers to as the {\it{residue}}) is:
\begin{equation*}
n-\sum_{i=1}^{n-1} (n-i) c_i p^iq^{i-1},
\end{equation*}
where $c_i = \tfrac{1}{i}\binom{2(i-1)}{i-1}$ is the $i^{\text{th}}$ Catalan number.\footnote{We use the Catalan numbers as defined in \cite{knuth}, whereas the modern convention is to refer to $c_{n+1}$ as the $n$th Catalan number.}
This setting was further considered by Stirzaker~\cite{stirzaker} wherein he elegantly derived Knuth's results by framing the problem as a simple random walk on the integers.
In addition to the expected value, Stirzaker was able to give a closed form for the probability generating function and extend the results to the case where the two matchboxes can initially contain a 
different number of matches.
Separately, Duffy and Dukes~\cite{duffydukes} studied the large deviations of the residue.
The limiting behaviour of the processes is discussed in Wallner~\cite[Sec. 3.5]{wallner}.
Related work includes the papers of  Guti\'errez and Subercaseaux~\cite{gs}, Cacoullos~\cite{cacoullos}, Lengyel~\cite{lengyel}, Stadje~\cite{stadje}, and Keane et al.~\cite{keane}.
In addition to this, we find it interesting to note that Prodinger~\cite{prodinger} used the kernel method to derive the generating function for the expected residue that was originally derived in \cite{knuth}.

In this paper we consider Knuth's generalization of Banach's matchbox problem for the case of $k$ matchboxes.
The $k$ matchboxes each initially contain $n$ matches. 
People arrive in succession and choose a match from one of the boxes. 
There are two types of people: {\it{big-choosers}} and {\it{little-choosers}}. 
A big-chooser takes a match from a box with the most number of matches remaining, 
while a little-chooser takes a match from a box with the least number remaining.
In this context, it is not significant how one deals with the event that there are multiple boxes from which a big- or little-chooser may choose.
We suppose that the probability of a person being a big-chooser is $p \in (0,1)$ and 
the probability of a person being a little-chooser is $q=1-p$.
Given this setup, we will consider the expected (aggregate) number of matches in the $k-1$ non-empty boxes at the moment that the remaining box first becomes empty, and we will call this aggregate quantity the {\it{residue}}, in keeping with Knuth's terminology.

The state of the $k$ matchboxes any time is represented by a $k$-tuple $(a_1,a_2,\ldots,a_k)$ with $a_1\geq a_2\geq \ldots  \geq a_k$.
The approach we take is more closely aligned with Knuth's than with Stirzaker's.
While Stirzaker's solution was arguably more succinct, unfortunately for us one of the key features of a simple random walk which Stirzaker was able to exploit, 
namely that a walk from points $a$ to $b$ must visit every intervening point en route, is not true of more general walks where the steps to the left and right are not of equal length. 
Indeed, if a walk starts at point $a$ and takes steps of $+2$ (with probability $q$) and $-1$ (with probability $1-q$), 
then it is possible (for non-trivial $q$) for walks to avoid any specified intervening point en route to $b$.

In Theorem~\ref{gf:one} we determine the generating function for the expected residue.
Interestingly, this generating function is a quotient whose denominator contains a generating function for a special case of the Raney numbers. 
The form for this generating function allows us to give, in Proposition~\ref{twofive}, an expression for the expected residue in terms of a sum that involves {\em diagonal state} return probabilities, 
where a diagonal state is a matchbox configuration in which all matchboxes each contain the same number of matches.
We use analytic techniques to determine, in Theorem~\ref{asymptotics_theorem}, the asymptotic behaviour of this expected residue for all values of $p$, 
and this involves the study of an asymmetric random walk.

In addition to this, in Theorem~\ref{thm:expectedreturn}, we present both the generating function and an expression for the expected order of the first return to a diagonal state.
The asymptotic behaviour of this quantity is determined in Proposition~\ref{five:asympt}.
The coefficients of the diagonal state probability generating function are shown to be related to manila folder configurations in a filing cabinet, and we make this connection precise in Section~\ref{sec:five:two}. 
This allows us to use the known enumeration of such manila folder configurations to give a closed form expression for the diagonal state probabilities.

\section{A recursion and diagonal-state numbers} \label{sec:two}
\subsection{A recursion for the expected residue} 

Let $X^{(k)}_{n}(p)$ be the random variable that is the sum of the number of matches in the $k-1$ larger boxes at the moment the $k$th box empties, as defined in Section~\ref{background}.
We refer to this quantity as the {\it{residue}}.
Let $M^{(k)}_{n}(p)$ be the expected value of $X^{(k)}_{n}(p)$, i.e. the expected residue.
Notice that $X^{(k)}_{n}(p)$ takes values in the interval $[k-1,(k-1)n]:=\{k-1,\ldots,(k-1)n\}$, and so $M^{(k)}_{n}(p)$ also takes values in this range.
For the case $k=3$, the first few values are easy to calculate.
The value $M^{(3)}_{1}(p) = 2$ since the first person empties the third box regardless of what type of chooser they are. 

In considering $X^{(3)}_{2}{(p)}$, the state of the system goes from $(2,2,2) \to (2,2,1)$ with probability 1, 
and we represent this by writing $(2,2,2) \toprob{1} (2,2,1)$.
From there, $(2,2,1) \toprob{q} (2,2,0)$ so that $\P(X^{(3)}_{2}(p)=4)=q$, and alternatively $(2,2,1) \toprob{p} (2,1,1)$.
Subsequently, there are two possibilities for what happens until we have an empty box:
$(2,1,1) \toprob{p} (1,1,1) \toprob{1} (1,1,0)$, which results in $X^{(3)}_{2}(p)=2$, and 
$(2,1,1) \toprob{q} (2,1,0)$, which results in $X^{(3)}_2(p)=3$. 
It follows that $M^{(3)}_{2}(p) = 4q+p(2p+3q) = 4-p-p^2$. 
Figure~\ref{diag:m3} illustrates the experimental value of $M^{(3)}_{100}(p)$ as a function of $p$.

During the time up until a box first empties, a time that we will to refer to as the {\em first phase}, 
the expected value can be calculated by recursions very similar to those stated by Knuth for the two-box problem. 
We have $M^{(k)}_n(p) = \Mean^{(k)}_{(k-1)n,n}(p)$ where $\Mean^{(k)}_{a,b}(p)$
is the expected value of the residue, given that the $k-1$ larger boxes collectively contain $a$ matches and the smallest box contains $b$ matches. 
We may calculate $\Mean^{(k)}_{a,b}(p)$ recursively using:
\begin{align} 
\label{sr1} \Mean^{(k)}_{a,0}(p) &= a;\\
 \label{sr2} \Mean^{(k)}_{(k-1)a,a}(p) &= \Mean^{(k)}_{(k-1)a,a-1}(p) & \mbox{ if } a>0; \\
\Mean^{(k)}_{a,b}(p) &= p \cdot \Mean^{(k)}_{a-1,b}(p) + q\cdot \Mean^{(k)}_{a,b-1} & \mbox{ if }a>(k-1)b>0. \label{sr3}
\end{align}
The above recursions have straightforward explanations. The first is trivial. 
The second holds since $(a,a,\ldots,a,a) \toprob{1} (a,a,\ldots,a,a-1)$. 
The third holds by conditional expectation.

\begin{figure}[!h]
\begin{center}
\includegraphics[scale=0.75]{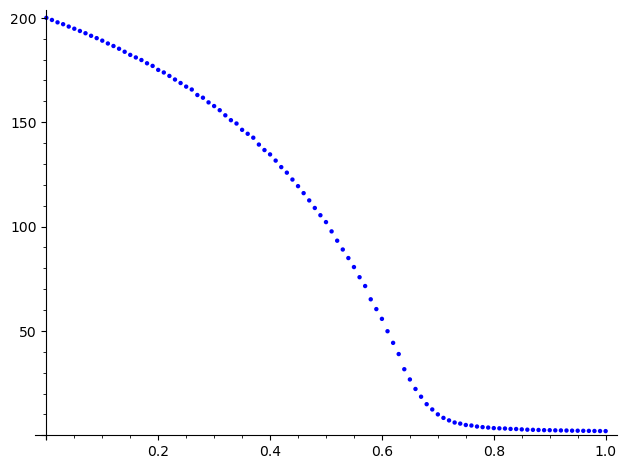}
\end{center}
\caption{The experimental value of $M^{(3)}_{100}(p)$. The horizontal axis represents the value of $p$ while the vertical axis represents $M^{(3)}_{100}(p)$.\label{diag:m3}}
\end{figure}

The solution to the recursions above depends on some numbers and power series that feature throughout this paper, so we will first present those.

\subsection{Diagonal state returns and the Raney numbers} \label{sec:diagonal} \label{sec:three}
We call a state in which all $k$ matchboxes have the same number of matches a {\em diagonal state}.
Consider the event that there are no diagonal states encountered in between the diagonal states 
$(n,n,\ldots,n)$ and $(n-i,n-i,\ldots,n-i)$. We will refer to such an event as a {\em first diagonal return of order $i$}, or simply a {\em first diagonal return}.
Let us denote by $s^{(k)}_i$ the number of ways in which this can happen.

The probability weight associated with a first diagonal return of order $i$ is $p^{(k-1)i} q^{i-1} = \tfrac{1}{q} (p^{k-1}q)^i$.
Notice that up until the first box empties, the difference in size between the largest and second smallest boxes can be no greater than one. 
The reason for this behaviour is that big-choosers have a leveling effect on the $(k-1)$ non-smallest boxes.
By considering the aggregate of the $k-1$ larger boxes, $s^{(k)}_i$ is equal to the number of lattice paths 
from $((k-1)n,n)$ to $((k-1)(n-i),n-i)$ that take unit south and west steps, and lie strictly below the 
$$\mbox{boundary line: }(k-1)y-x=0,$$ except at the two endpoints.

Paths through the first quadrant of the lattice that are weakly bounded above by the boundary line 
are equivalent to ballot problems. 
Dvoretzky and Motzkin~\cite{dvor} established the following result about a ballot in which $a$ votes are cast
for candidate $A$ and $b$ votes are cast for candidate $B$:
\medskip

\begin{center}
\begin{minipage}{0.8\textwidth}
\it{If $\alpha$ is an integer satisfying $0 < \alpha < \frac{a}{b}$ , 
then the probability that throughout the counting the number
of votes registered for $A$ is always greater than $\alpha$ times the number 
of votes registered for candidate $B$ equals $\frac{a-\alpha b}{a+b}$.}
\end{minipage}
\end{center}

In our case, any path from $((k-1)n, n)$ to $((k-1)(n-i), n-i)$ lying below the given boundary line must start with a step from $((k-1)n, n)$ to $((k-1)n,n-1)$. 
There are $\binom{(k-1)i + i-1 }{i-1} = \binom{ki-1}{i-1}$ unbounded paths through the lattice from $((k-1)n,n - 1)$ to $((k-1)(n-i), n-i)$.
Dvoretzky and Motzkin's theorem tells us, by setting {$a = (k-1)i$, $b = i-1$ and $\alpha = k-1$}, 
that the probability of choosing one that lies strictly below $(k-1)y - x = 0$ is $k-1/(ki-1)$.
Thus the number of such paths is 
$$s^{(k)}_i =\frac{k-1}{ki-1}\binom{ki-1}{i-1} = \frac{k-1}{k(i-1)+k-1}\binom{k(i-1)+k-1}{i-1}.$$

This form of a binomial expression is well-known in enumerative combinatorics and represents a special case of the Raney numbers~\cite{raney}:
$$\Raney_{c,r}(\ell) = \dfrac{r}{\ell c+r} \binom{\ell c+r}{\ell}.$$
These numbers have several combinatorial interpretations and are also sometimes referred to as the two-parameter Fuss-Catalan numbers. 
For example, 
$\Raney_{c,r}(\ell)$
is the number of sequences $(a_1,\ldots,a_{\ell c +r})$ where $a_i \in \{1,1-c\}$ with all partial sums $a_1+\ldots+a_j$ positive and $a_1+\ldots+a_{\ell c +r}=r$.  
The numbers in which we are interested are the special case:
$$s^{(k)}_i = \Raney_{k,k-1}(i-1).$$
That our numbers are instances of the Raney numbers allows us to make use of identities that they satisfy. 
Interesting aspects of these numbers can be found in Graham et al.~\cite[Sec. 7.5, Example 5]{concretemath}, Hilton and Pedersen~\cite{hilton}, and Zhou and Yan~\cite{raney2}.

Define the generating function
$$S^{(k)}(z) := \sum_{n\geq 1} s^{(k)}_n z^n.$$
It is known, see for example Liszewska and M\l otkowski~\cite[Sec. 9.2]{LM},
that $S^{(k)}(z)$ satisfies 
\begin{equation}\label{sgeneratingfunction}
S^{(k)}(z) = \dfrac{z}{(1-S^{(k)}(z))^{k-1}}.
\end{equation}
The first few such power series are:
\begin{align*}
S^{(2)}(z) &= \sum_{n\geq 1} \frac{1}{n}\binom{2(n-1)}{n-1} z^n = z + z^2 + 2z^3 + 5z^4 + 14z^5+42z^6+\ldots \\
S^{(3)}(z) &= \sum_{n\geq 1} \frac{1}{n}\binom{3(n-1)+1}{n-1}z^n    =z+2z^2+7z^3+ 30z^4 + 143z^5+728z^6 + \ldots\\
S^{(4)}(z) &= \sum_{n\geq 1} \frac{1}{n}\binom{4(n-1)+2}{n-1}z^n =  z+ 3z^2 + 15z^3+ 91z^4+ 612z^5 + 4389z^6 + \ldots
\end{align*}
The coefficients of $S^{(2)}(z)$ are the Catalan numbers and $S^{(2)}(z) = (1-\sqrt{1-4z})/2$. An exact expression for $S^{(3)}(z)$ is also known 
$$S^{(3)}(z) = \dfrac{4}{3} \sin^2\left( \frac{1}{3} \arcsin \sqrt{\dfrac{27z}{4}} \right),$$
see e.g. Gessell~\cite[Lemma 3.2]{gessell}.
Both are easily seen to satisfy equation (\ref{sgeneratingfunction}).

As we mentioned above, the probability associated with one of the paths counted by $s^{(k)}_i$ is $p^{(k-1)i}q^{i-1}$, 
and so we will be interested in
the probability generating function
 $$\sum_{i\geq 1} s^{(k)}_i p^{(k-1)i}q^{i-1}z^i = \tfrac{1}{q} S^{(k)}(p^{k-1}qz).$$
The quotient $1/(1-\frac{1}{q}S^{(k)}(p^{k-1}qz))$ will also appear frequently in this paper so we discuss it briefly here. 
Let us write
\begin{align}
f^{(k)}(z) 	= f_0^{(k)}+f_1^{(k)}z+f_2^{(k)}z^2+\dots 
			:= \dfrac{1}{1-\frac{1}{q}S^{(k)}(p^{k-1}qz)}. \label{f:defn}
\end{align}
The function $\frac{1}{q}S^{(k)}(p^{k-1}qz)$ is the probability generating function for the order of the first return to a diagonal state, 
i.e. the coefficient of $z^i$ is the probability that the first return to a diagonal state occurs after $ki$ steps.
Consequently the quotient
$1/\left(1-\frac{1}{q}S^{(k)}(p^{k-1}qz)\right)$ is the probability generating function for 
the event of being in a diagonal state after some multiple of $k$ steps, but not necessarily for the first time.
See e.g. Grimmett and Stirzaker's book~\cite[Sec. 5.3, Thm. 1(a)]{gsbook}. 

Given the discussion so far about first-quadrant lattice paths from $((k-1)n,n)$ to the origin (and other points) that lie weakly below the boundary line,
it would be sensible to introduce some terminology related to this.
Let $\Paths_k(n)$ be the set of all those paths from $((k-1)n,n) \mapsto (0,0)$ that take unit south and west steps and lie in the first quadrant weakly below the boundary line
$(k-1)y-x=0$.
Let $\Paths_k(n,i)$ be the subset of those paths in $\Paths_k(n)$ that return to the boundary line precisely $i$ times between the two endpoints.
Let $\Paths_k : = \{\epsilon\} \cup_{n \geq 1} \Paths_k(n)$, where $\epsilon$ is the empty path.
Notice that members of all of the above sets can be represented by words in $\{D,L\}^{*}$ where $D$ corresponds to a unit south step and $L$ corresponds to a unit west step.
Finally, let $\paths_k(n) = |\Paths_k(n)|$ and $\paths_k(n,i) = |\Paths_k(n,i)|$.

\section{The expected residue}
We are now in a position to solve the recurrence stated in equations (\ref{sr1})--(\ref{sr3}).
First we will prove a recursive expression for the expected residue $M^{(k)}_n(p)$.
We will then use that recursion to give an expression for the generating function $M^{(k)}(z)$ for the numbers $\left(M^{(k)}_n(p)\right)_{n\geq 1}$.
The generating function has a pleasing form and it is possible to extract the coefficient of $z^n$ in order to give a non-recursive expression
for $M^{(k)}_n(p)$ that involves only numbers that appear in Subsection~\ref{sec:three}.

\begin{lemma}\label{lemma:one}
$M^{(k)}_1(p)=k-1$.
For all $n\geq 2$ we have
\begin{align}\label{recursion:one}
M^{(k)}_n(p) = 
	\frac{1}{q} \sum_{0<i<n} s^{(k)}_i  (p^{k-1}q)^i M^{(k)}_{n-i}(p)
	~+~ L^{(k)}_n(p),
\end{align}
where
\begin{align*}
L^{(k)}_n(p) &= (p^{k-1}q)^n \frac{1}{q}\displaystyle\sum_{j=k}^{(k-1)n} j d^{(k)}_{n,j}({1}/{p})^j
\end{align*}
and 
$$d^{(k)}_{n,j}=
\begin{cases}
\displaystyle\binom{kn-2-j}{n-2} - \sum_{\ell=1}^{n-\left\lceil \frac{j}{k-1} \right\rceil} s^{(k)}_{\ell} \cdot \binom{k(n-\ell)-1-j}{n-\ell-1} & \mbox{ if $ k \leq j \leq (k-1)n$} \\
0 & \mbox{ otherwise.}
\end{cases}
$$
\end{lemma}

\begin{proof}
The recurrences in equations (\ref{sr1})--(\ref{sr3}) provide a means to calculate $M^{(k)}_n(p)$ for general $n$. 
The expression for $M^{(k)}_1(p)$ follows directly from these.
One systematic way to evaluate $M^{(k)}_n(p) = M^{(k)}_{(k-1)n,n}(p)$ for $n\geq 2$ is to apply the recursion to 
all terms except those terms of the form $\Mean^{(k)}_{(k-1)a,a}(p)$ for $a<n$. 
For example we can write:
\begin{align*}
M^{(3)}_2(p) &=  p^2 M^{(3)}_1(p) + (3pq + 4q) \\
M^{(3)}_3(p) &=  p^2 M^{(3)}_{2}(p) + 2p^4q  M^{(3)}_{1}(p)  ~ + ~ (6p^3q^2  +8p^2q^2 + 10pq^2  + 6 q^2).
\end{align*}
Doing this for general $n$ will result in an expression for $M^{(k)}_n(p)$ of the following type:
\begin{align*}
M^{(k)}_n(p) = \Mean_{(k-1)n,n}(p) &= \sum_{i=1}^{n-1} \alpha^{(k)}_{n,i}(p) M_{(k-1)(n-i),n-i}(p) ~+~ \beta^{(k)}_n(p) \\
&= \sum_{i=1}^{n-1} \alpha^{(k)}_{n,i}(p) M^{(k)}_{n-i}(p) ~+~ \beta^{(k)}_n(p),
\end{align*}
where $\alpha^{(k)}_{n,i}(p)$ and $\beta^{(k)}_n(p)$ are yet-to-be determined functions.
One may easily observe that, due to the form of the recursions, $\alpha_{n,i}(p)$ is a function of $i$, $k$, and $p$, but not of $n$.
The quantity $\alpha_{n,i}(p)$ is the probability of a chooser process that returns to a `diagonal' state $(x,x,\ldots,x)$, 
for the first time, at $(n-i,n-i,\ldots,n-i)$. From Subsection~\ref{sec:diagonal} this probability is
$\alpha^{(k)}_i(p) = s^{(k)}_i (p^{k-1}q)^i / q$.
Thus we have 
\begin{align}\label{sofarsogood}
M^{(k)}_n(p) 
&= \tfrac{1}{q}\sum_{i=1}^{n-1} s^{(k)}_i(p^{k-1}q)^i M^{(k)}_{n-i}(p) ~+~ \beta^{(k)}_n(p).
\end{align}

The first term on the right hand side of equation (\ref{sofarsogood}) represents the expected residue amongst those chooser processes 
that cause a return to some diagonal state $(n-i,n-i,\ldots,n-i)$ before the first matchbox empties. 
The second term represents those that do not encounter a diagonal state.
Therefore we can write 
$$\beta^{(k)}_n(p) = \sum_{j} j d^{(k)}_{n,j} p^{(k-1)n-j} q^{n-1}$$
where $d^{(k)}_{n,j}$ is the number of lattice paths from $((k-1)n,n-1) \to \cdots \to (j,1) \to (j,0)$ that do not touch the boundary line $(k-1)y-x=0$.

Consider the set of all paths from $((k-1)n,n-1) \to \cdots \to (j,1) \to (j,0)$.
The number of such paths is $\binom{((k-1)n-j) + (n-2)}{n-2} = \binom{kn-2-j}{n-2}$.
Let us call a path in this set {\it{good}} if it does not touch the boundary line $(k-1)y-x=0$, and {\it{bad}} otherwise.
Every bad path is characterized by the path returning to the boundary line before reaching $(j,1)$. 
Suppose this happens for a bad path for the first time at $((k-1)(n-\ell),n-\ell)$ for some $\ell\geq 1$. 
The number of ways a path can go from $((k-1)n,n-1) \to \cdots \to ((k-1)(n-\ell),n-\ell)$ is $s_{\ell}$ and, since it is a bad path, 
there are no further restrictions on the steps from $((k-1)(n-\ell),n-\ell)$ to $(j,1)$. 
That latter number is $\binom{(k-1)(n-\ell)-j + n-\ell-1}{n-\ell-1} = \binom{k(n-\ell)-j-1}{n-\ell-1}$. 
Therefore the number of good paths is 
$$d^{(k)}_{n,j} = \binom{kn-2-j}{n-2} - \sum_{\ell \geq 1} s^{(k)}_{\ell} \cdot  \binom{(k-1)(n-\ell)-j-1}{n-\ell-1}.\qedhere$$
\end{proof}

Let 
\begin{align*}
M^{(k)}(z) = \sum_{n\geq 1} M^{(k)}_n(p) z^n  \qquad \mbox{ and } \qquad
L^{(k)}(z) = \sum_{n\geq 2} L^{(k)}_n(p) z^n. 
\end{align*}
Note the starting index of $L^{(k)}(z)$ is 2.
This is because $L_n^{(k)}(p)$ denotes the contribution to $M^{(k)}_n(p)$ of paths which have no diagonal return, and there can be no such paths in the case $n=1$.

\begin{lemma}\label{expression:for:L}
Let $k\geq 2$. Then 
$$L^{(k)}(z) =
 \displaystyle \dfrac{z(kp-(k-1)) S^{(k)}(p^{k-1}qz) - qz^2((2k-1)p-(2k-2)+(k-1)(qz))}{q^2(1-z)^2}. $$
\end{lemma}

\begin{proof}
By definition,
\begin{align*}
L^{(k)}_n(p) & = \sum_{j=k}^{(k-1)n}  j d^{(k)}_{n,j} p^{(k-1)n-j} q^{n-1} 
 = \frac{(p^{k-1}q)^n}{pq} \sum_{j=k}^{(k-1)n}  j d^{(k)}_{n,j} (1/p)^{j-1}.
\end{align*}
The generating function for these is then
\begin{align*}
L^{(k)}(z) &= \sum_{n\geq 2} L^{(k)}_n(p) z^n \\
& =  \sum_{n\geq 2}z^n \left( \frac{(p^{k-1}q)^n}{pq} \sum_{j=k}^{(k-1)n}  j d^{(k)}_{n,j} (p^{-1})^{j-1} \right)  \\
& =  \frac{1}{pq} \sum_{n\geq 2} \left(\sum_{j=k}^{(k-1)n}  j d^{(k)}_{n,j} (p^{-1})^{j-1}\right)  (p^{k-1}qz)^n  \\
& =  \frac{1}{pq} \dfrac{\partial D^{(k)}(x,y)}{\partial x}\Bigg|_{x\leftarrow 1/p \atop y \leftarrow p^{k-1}qz}
\end{align*}
where $D^{(k)}(x,y) :=\sum_{n\geq 2} D^{(k)}_n(x) y^n$ and $D^{(k)}_n(x) := \sum_{j=k}^{(k-1)n}  d^{(k)}_{n,j} x^{j}$.

Any path to $(j,1)$ either passes through $(j,2)$ or through $(j+1,1)$, but not both.
It follows that $d^{(k)}_{n,j} = d^{(k)}_{n,j+1} + d^{(k)}_{n-1,j-(k-1)}$ holds for those values of $n$ and $j$ for which all quantities in the recursion are defined.
It also turns out easier, from a generating function perspective, to reverse the sequence $(d^{(k)}_{n,j})_j$. 
Set $t^{(k)}_{n,j} := d^{(k)}_{n,(k-1)n-j}$ for all 
$0\leq j \leq (k-1)n-k$ and define $T^{(k)}_n(x) := \sum_{j=0}^{(k-1)n-k} t^{(k)}_{n,j} x^j$. 
Note that $D^{(k)}_n(x) = x^{(k-1)n} T_n(1/x)$.
For example, we have
\begin{align*}
T^{(3)}_2(x) &= 1+x \\
T^{(3)}_3(x) &= 1+2x+2x^2+2x^3 \\
T^{(3)}_4(x) &= 1+3x+5x^2+7x^3+7x^4+7x^5\\
T^{(3)}_5(x) &= 1+4x+9x^2+16x^3+23x^4+30x^5+30x^6+30x^7.
\end{align*}
Notice that the expression for $T^{(k)}_2(x)$ for a given value of $k$ is
$$T^{(k)}_2(x) = 1+x+x^2+\ldots+x^{k-2}.$$
The recurrence for $d^{(k)}_{n,j}$ above can be transformed into one for $t^{(k)}_{n,j}$: 
$$t^{(k)}_{n,j}=t^{(k)}_{n,j-1}+t^{(k)}_{n-1,j}.$$ 
From this it is clear that the sequence $(t^{(k)}_{n,j})$, as a sequence over $j$, is the sequence of all partial sums of $(t^{(k)}_{n-1,j})$, 
with the largest term appended to the end of the new sequence $k-1$ times. 
This is because $t^{(k)}_{n,j}$ is only defined for $j\leq (k-1)n-k$.
Notice that the largest value in the sequence $t^{(k)}_{n+1,j}$ for $n$ fixed is $s^{(k)}_n$.
See Figure~\ref{andrews:diagram} for an illustration of why this is the case.
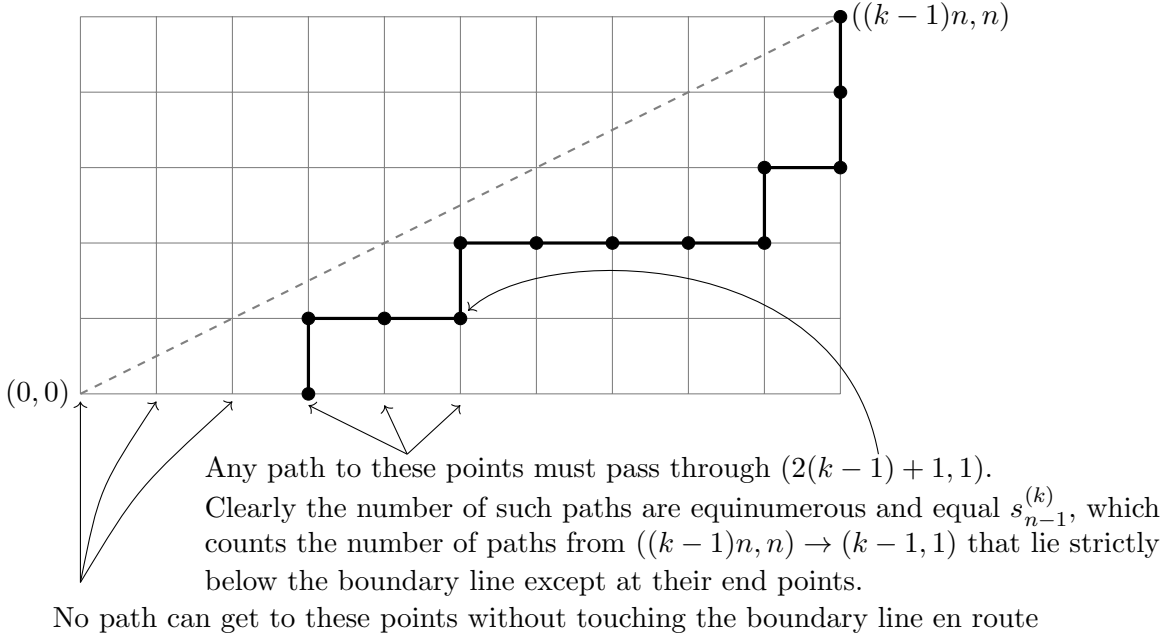
\begin{figure}
\begin{tikzpicture}
\draw[gray,thin] (0,0) grid (10,5);
\draw[very thick] (10,5) -- (10,3) -- (9,3) -- (9,2) -- (8,2) -- (5,2) -- (5,1) -- (3,1) -- (3,0);
\draw[gray,thick,dashed] (0,0) -- (10,5);
\foreach \x/\y in {10/5,10/4,10/3,9/3,9/2,8/2,7/2,6/2,5/2,5/1,4/1,3/1,3/0}{
    \draw[fill=black] (\x,\y) circle (0.5ex);
    }
\node[anchor=east] at (0,0) {$(0,0)$};
\node[anchor=west] at (10,5) {$((k-1)n,n)$};
\node[anchor=west] at (1.5,-1) {Any path to these points must pass through $(2(k-1)+1,1)$.};
\node[anchor=west] at (1.5,-1.5) {Clearly the number of such paths are equinumerous and equal $s^{(k)}_{n-1}$, which};
\node[anchor=west] at (1.5,-2) {counts the number of paths from $((k-1)n,n) \to (k-1,1)$ that lie strictly};
\node[anchor=west] at (1.5,-2.5) {below the boundary line except at their end points.};
\node[anchor=west] at (-0.5,-3) {No path can get to these points without touching the boundary line en route};
\draw[->]        (0,-2.5) to (0,-0.1);
\draw[->]        (0,-2.5)  .. controls (0.5-0.25,-1.25) ..  (1,-0.1);
\draw[->]        (0,-2.5)  .. controls (0.5+0.5-0.25,-1.25) ..  (2,-0.1);
\draw[->]        (4.3,-0.8) to (3,-0.15);
\draw[->]        (4.3,-0.8) to (4,-0.15);
\draw[->]        (4.3,-0.8) to (5,-0.15);
\draw[->]        (10.5,-0.8)  .. controls (10,2) and (6,2) ..  (5.1,1.1);
\end{tikzpicture}
\caption{Paths counted by the numbers $d^{(k)}_{n,j}$: those paths from $((k-1)n,n)$ to $(j,0)$ that do not touch 
the boundary line or the $x$-axis en-route. This diagram relates to the proof of Lemma~\ref{expression:for:L}. 
Here $k=3$.
\label{andrews:diagram}} 
\end{figure}
In terms of $T^{(k)}_{n+1}(x)$ and $T^{(k)}_n(x)$ this is 
$$T^{(k)}_{n+1}(x)  = \dfrac{T^{(k)}_n(x)}{1-x} - \dfrac{x^{(k-1)n} s^{(k)}_n}{1-x}.$$
This recursion now allows us to give an expression for
\begin{align*}
T^{(k)}(x,y) := \sum_{n\geq 2} T^{(k)}_n(x) y^n 
&= T^{(k)}_2(x)y^2 + \sum_{n\geq 2} T^{(k)}_{n+1}(x) y^{n+1} \\
&= T^{(k)}_2(x)y^2 + \sum_{n\geq 2} y^{n+1} \left( \dfrac{T^{(k)}_n(x)}{1-x} - \dfrac{x^{(k-1)n} s^{(k)}_n}{1-x} \right) \\
&= T^{(k)}_2(x)y^2 + \dfrac{x^{k-1} y^2}{1-x} - \dfrac{yS^{(k)}(x^{k-1}y)}{1-x} + \dfrac{y}{1-x} T^{(k)}(x,y).
\end{align*}
Rearranging gives
\begin{align*}
T^{(k)}(x,y) \left(1-\dfrac{y}{1-x}\right) 
	&=  \dfrac{(1-x)T^{(k)}_2(x)y^2 + {x^{k-1} y^2} - yS^{(k)}(x^{k-1}y)}{1-x}\\
	&=  \dfrac{(1-x^{k-1})y^2 + {x^{k-1} y^2} - yS^{(k)}(x^{k-1}y)}{1-x},
\end{align*}
and so
$$T^{(k)}(x,y) =  \dfrac{y^2  - yS^{(k)}(x^{k-1}y)}{1-x-y}. $$
Now since $D^{(k)}_n(x) = x^{(k-1)n} T_n(1/x)$ we have
\begin{align}
D^{(k)}(x,y) &= \sum_{n\geq 2} D^{(k)}_n(x) y^n \nonumber \\
	&= \sum_{n\geq 2} x^{(k-1)n} T_n(1/x) y^n \nonumber \\
	& = T(1/x,x^{k-1}y) \nonumber \\
	& =  \dfrac{(x^{k-1}y)^2  - (x^{k-1}y) S^{(k)}((x^{-1})^{k-1} (x^{k-1}y))}{1-x^{-1}-x^{k-1}y} \nonumber \\
	& =  \dfrac{x^{k}y S^{(k)}(y)-x^{2k-1}y^2}{1+x^{k}y-x} .\label{Dformula}
\end{align}
This gives
\begin{align*}
\dfrac{\partial D^{(k)}(x,y)}{\partial x}
	=& \dfrac{ (1+x^ky-x) (kx^{k-1}yS^{(k)}(y) - (2k-1)x^{2k-2} y^2)  }{(1+x^ky-x)^2}\\ & -\dfrac{(kx^{k-1}y -1) (x^kyS^{(k)}(y)-x^{2k-1}y^2)}{(1+x^ky-x)^2} \\
	=& \dfrac{x^{k-1}(k-(k-1)x)yS^{(k)}(y) }{(1+x^ky-x)^2} \\ & - \dfrac{x^{2k-2}y^2((2k-1)-(2k-2)x+(k-1)x^ky)}{(1+x^ky-x)^2}.
\end{align*}
Using this expression it follows that 
\begin{align*}
 L(z) &= \frac{1}{pq} \dfrac{\partial D^{(k)}(x,y)}{\partial x}\Bigg|_{x\leftarrow 1/p \atop y \leftarrow p^{k-1}qz}
		= \frac{1}{pq} \dfrac{\partial D(x,p^{k-1}qz)}{\partial x} \bigg|_{x\leftarrow 1/p} \\
	&= \dfrac{(p^{-1})^{k-1}(k-(k-1)p^{-1})p^{k-1}qzS^{(k)}(p^{k-1}qz)}{pq(1+(p^{-1})^k(p^{k-1}qz)-p^{-1})^2} \\
	&  \qquad	- \dfrac{(p^{-1})^{2k-2}(p^{k-1}qz)^2((2k-1)-(2k-2)p^{-1}+(k-1)(p^{-1})^k(p^{k-1}qz))}{pq(1+(p^{-1})^k(p^{k-1}qz)-p^{-1})^2}\\
	&= \dfrac{z(kp-(k-1)) S^{(k)}(p^{k-1}qz) - qz^2((2k-1)p-(2k-2)+(k-1)(qz))}{q^2(1-z)^2}.\qedhere
\end{align*}
\end{proof}

Having worked through these intermediary lemmas, we are now ready to state the main theorem concerning the generating function for the sequence of expected residues.
\begin{theorem}\label{gf:one}
$M^{(k)}(z) = \dfrac{z}{(1-z)^2} \cdot \left( k - \dfrac{1}{q}\right)  + \dfrac{z}{1-z} \cdot \dfrac{p}{q} \cdot \dfrac{1}{\left(1-\tfrac{1}{q} S^{(k)}(p^{k-1}qz)\right)}$.
\end{theorem}

\begin{proof}
Multiply both sides of the equation in Lemma~\ref{lemma:one} by $z^n$ and sum over all $n\geq 2$ to get
$$M^{(k)}(z)-M^{(k)}_1(p)z = \frac{1}{q}S^{(k)}(p^{k-1}qz)M^{(k)}(z)+ L^{(k)}(z).$$
Rearrange this, while also using the fact that $M^{(k)}_1(p)=k-1$, to get 
$$M^{(k)}(z) = \dfrac{(k-1)z+L^{(k)}(z)}{1-\tfrac{1}{q}S^{(k)}(p^{k-1}qz)}.$$
Lemma~\ref{expression:for:L} provides an expression for $L^{(k)}(z)$ and so the numerator above equals
\begin{align*}
(k-1)z+L^{(k)}(z) 
 =~ & (k-1)z +  \dfrac{z(kp-(k-1)) S^{(k)}(p^{k-1}qz)}{q^2(1-z)^2} \\  
		& - \dfrac{qz^2((2k-1)p-(2k-2)+(k-1)({p}qz))}{q^2(1-z)^2}\\
 =~ & \dfrac{z}{q^2(1-z)^2}\left((pk-(k-1))S_k(p^{k-1}qz) -pqz + (k-1)q^2\right).
\end{align*}
Substituting this back into the expression for $M^{(k)}(z)$ gives:
\begin{align*}
M^{(k)}(z) 
	&=   \dfrac{z\left((pk-(k-1))S_k(p^{k-1}qz) -pqz + (k-1)q^2\right)}{q^2(1-z)^2\left( 1-\frac{1}{q}S^{(k)}(p^{k-1}qz) \right)}\\
	&=  \dfrac{z(kq-1)}{q(1-z)^2} + \dfrac{p z}{q(1-z) \left(1-\tfrac{1}{q} S^{(k)}(p^{k-1}qz)\right)}.\qedhere
\end{align*}
\end{proof}

From Theorem~\ref{gf:one} we now deduce a succinct expression for the expected residue $M^{(k)}_n(p)$.
\begin{proposition} \label{twofive}
For all $n\geq 1$,
\begin{equation}
M^{(k)}_n(p)= 
 (k-1)n - \frac{p}{q}\left( (1-f^{(k)}_0)+(1-f^{(k)}_1)+\ldots + (1- f^{(k)}_{n-1}) \right),
\end{equation}
where $f^{(k)}_i$, defined in equation (\ref{f:defn}), is the probability of being in a diagonal state after $ki$ steps.
\end{proposition}

\begin{proof}
Extract the coefficient of $z^n$ in $M^{(k)}(z)$ from Theorem~\ref{gf:one} to find
\begin{align*}
M_n(p) &= [z^n]   \dfrac{kq-1}{q} \cdot \dfrac{z}{(1-z)^2} + [z^n] \dfrac{p}{q} \cdot \dfrac{z}{(1-z) \left(1-\tfrac{1}{q} S^{(k)}(p^{k-1}qz)\right)}\\
&=  \dfrac{(k-1)q-p}{q} \cdot [z^n]\dfrac{z}{(1-z)^2} +   \dfrac{p}{q} \cdot [z^n]\dfrac{z}{(1-z) \left(1-\tfrac{1}{q} S^{(k)}(p^{k-1}qz)\right)}\\
&=   (k-1)n - n\frac{p}{q} +   \dfrac{p}{q} \sum_{i=0}^{n-1} [z^i] \dfrac{1}{1-\tfrac{1}{q} S^{(k)}(p^{k-1}qz)}\\
&=   (k-1)n - \dfrac{p}{q}\left( n - \sum_{i=0}^{n-1} [z^i] \dfrac{1}{1-\tfrac{1}{q} S^{(k)}(p^{k-1}qz)}\right).
\end{align*}
From Subsection~\ref{sec:diagonal} we have 
\begin{align*}
M^{(k)}_n(p)	& = (k-1)n - \frac{p}{q}\left( n- (f^{(k)}_0+f^{(k)}_1+\ldots + f^{(k)}_{n-1}) \right)\\
		& = (k-1)n - \frac{p}{q}\left( (1-f^{(k)}_0)+(1-f^{(k)}_1)+\ldots + (1- f^{(k)}_{n-1}) \right),
\end{align*}
where $f^{(k)}_i$ is the probability that a path starting from the diagonal state $(n,n,\ldots,n)$ reaches the diagonal state $(n-i,n-i,\ldots,n-i)$.
Note that it may well have gone through diagonal states in between.
\end{proof}

\section{Asymptotic analysis of the expected residue}
Now we turn our attention to the asymptotic behaviour of $M^{(k)}_n(p)$ for fixed $p$ as $n \rightarrow \infty$. 

\begin{theorem}\label{asymptotics_theorem}
    Let $R^{*}=\dfrac{(k-1)^{k-1}}{k^k} \cdot \dfrac{1}{p^{k-1}q}$ and let $\epsilon>0$. Then
	\begin{align*}
        M^{(k)}_n(p)=
        \begin{cases}
            \left(k-\frac{1}{q}\right)n + \dfrac{p}{q} \cdot \dfrac{1}{1-\lambda_{k,p}} + \mathcal{O} \left(\left(\dfrac{1}{R^{*}}+\epsilon\right)^n\right) & \mbox{ if }q > \dfrac{1}{k} \\[1em]
			\sqrt{\dfrac{2k(k-1)n}{\pi}}
				+ 
					\mathcal{O} \left(\dfrac{1}{n^{1/2}}\right)
						  & \mbox{ if }q = \dfrac{1}{k} \\[1em]
            \dfrac{(k-1)(2-kq)}{2(1-kq)} + \mathcal{O} \left(\left(\dfrac{1}{R^{*}}+\epsilon\right)^n\right) & \mbox{ if }q < \dfrac{1}{k},
        \end{cases}
	\end{align*}
    where $\lambda_{k,p}:=\frac{1}{q}S^{(k)}(p^{k-1}q) \in (0,1)$ when $q>1/k$.
\end{theorem}

The values of $\lambda_{k,p}$ for $k=2,3,4,$ and $5$ are shown in Figure~\ref{fig:lambdas}. 
We require several technical lemmas in order to prove Theorem~\ref{asymptotics_theorem}. 

\begin{figure}
\includegraphics[scale=0.8]{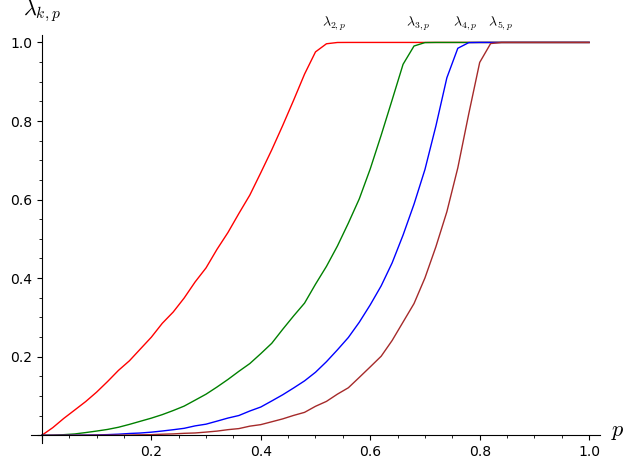}
\caption{
The value of $\lambda_{k,p}$ for $k \in \{2,3,4,5\}$. The lines for $\lambda_{2,p}$, $\lambda_{3,p}$, $\lambda_{4,p}$, and $\lambda_{5,p}$ are coloured 
red, green, blue, and brown, respectively.  
	\label{fig:lambdas}}
\end{figure}

\begin{lemma} \label{RadiusOfConvergence}
    $\frac{1}{q}S_k(p^{k-1}qz)$ converges for $|z|<\frac{(k-1)^{k-1}}{k^k}\frac{1}{p^{k-1}q}$.
\end{lemma}

\begin{proof}
Let $(\alpha_n)_{n \geq 1}$ be the sequence generated by $\frac{1}{q}S^{(k)}(p^{k-1}qz)$ as a power series in $z$, so that 
$$\alpha_n=\frac{1}{q}\cdot\frac{k-1}{kn-1}\binom{kn-1}{n-1}(p^{k-1}qz)^n.$$ 
Apply the ratio test to $(\alpha_n)_{n \geq 1}$ to find 
\begin{align*}
    \dfrac{|\alpha_{n+1}|}{|\alpha_n|} 
	& = \dfrac{\frac{1}{q}\frac{k-1}{k(n+1)-1}\binom{k(n+1)-1}{n}(p^{k-1}q|z|)^{n+1}}{\frac{1}{q}\frac{k-1}{kn-1}\binom{kn-1}{n-1}(p^{k-1}q|z|)^{n}}\\
    & =\left(\frac{kn-1}{kn+k-1}\right) p^{k-1}q|z| \cdot \frac{(kn+k-1)!}{((k-1)n+k-1)!n!} \cdot \frac{((k-1)n)!(n-1)!}{(kn-1)!}\\
    & =\left(\frac{kn-1}{kn+k-1}\right)p^{k-1}q|z|\cdot \frac{1}{n} \cdot \frac{(kn+k-1)(kn+k-2)\cdots(kn)}{((k-1)n+k-1)((k-1)n+k-2)\cdots((k-1)n+1)}\\
    & \rightarrow \quad \dfrac{k^k}{(k-1)^{k-1}}p^{k-1}q|z| \mbox{ as }n \rightarrow \infty.
\end{align*}
Therefore $\frac{1}{q}S^{(k)}(p^{k-1}qz)$ converges for $|z|<\frac{(k-1)^{k-1}}{k^k}\frac{1}{p^{k-1}q}$ and diverges for $|z|>\frac{(k-1)^{k-1}}{k^k}\frac{1}{p^{k-1}q}$.
\end{proof}
Notice that, on the interval $p \in (0,1)$, the function $p\mapsto p^{k-1}q$ has a maximum value of $\frac{(k-1)^{k-1}}{k^k}$ that occurs uniquely at $p=\frac{k-1}{k}$. 
Therefore $\frac{1}{q}S^{(k)}(p^{k-1}qz)$ is holomorphic on a disc centered at the origin of radius $R^{*}:=\frac{(k-1)^{k-1}}{k^k}\frac{1}{p^{k-1}q}$
and, moreover, $R^{*}>1$ for $q \neq \frac{1}{k}$.

\begin{lemma} \label{technical_lemma_one}
If $\frac{1}{q}S^{(k)}(p^{k-1}qz)=1$, then $z=1$.
\end{lemma}

\begin{proof}
By equation (\ref{sgeneratingfunction}), $S^{(k)}(p^{k-1}qz) = \frac{p^{k-1}qz}{(1-S^{(k)}(p^{k-1}qz))^{k-1}}$. Replacing $S^{(k)}(p^{k-1}qz)$ with $q$, we find:

$$q=\dfrac{p^{k-1}qz}{(1-q)^{k-1}}.$$
Rearranging gives $z=1$.
\end{proof}

\begin{lemma}\label{probeventualreturn}
We have
$$\lambda_{k,p} :=\frac{1}{q} S^{(k)}(p^{k-1}q)
	\begin{cases}
	=1 & \mbox{ if } q<\tfrac{1}{k} \\
	\in (0,1)  & \mbox{ if } q>\tfrac{1}{k}.
	\end{cases}
$$
\end{lemma}

\begin{proof}
Recall that the function $\frac{1}{q}S^{(k)}(p^{k-1}qz)$ is the probability generating function for the order of the first return to a diagonal state.
So $\frac{1}{q}S^{(k)}(p^{k-1}q)$ is the probability that an infinite walk will ever return to the boundary line. 
The initial step in our paths process is a south step and any return to the diagonal requires $k-1$ as many west steps as south steps. 
We will consider a related random walk in order to determine whether or not this will occur with probability $1$.

Consider a random walk $S_n=S_0+(X_1+\ldots +X_n)$ on $\mathbb{Z}$ that starts at $S_0=-(k-1)$ and takes steps $X_i \in \{-(k-1),1\}$ with 
\begin{align*}
\P(X_i=+1) &= p \\
\P(X_i=-(k-1)) &= q.
\end{align*}
The event $S_i=0$ for this random walk is equivalent to the event of a diagonal return in our paths process.

We are interested in the probability that $S_i=0$ for some $i>0$. As the $X_i$ are independent, 
the expectation 
\begin{align*}
\E(S_i) = & S_0 + \sum_{j=1}^i \E(X_j)\\
		 = & - (k-1) + i\E(X_1)  \\
		= & - (k-1) + i(1-qk),
\end{align*}
since $\E(X_1) = 1-qk$. 
Notice that $\E(X_1)>0$ for $q<1/k$ and $\E(X_1)<0$ for $q>1/k$. 

Suppose that $q<1/k$, and let $1-qk=\mu~$($\mu>0$). 
Let $\mathrm{var}(X_1)=\sigma^2$ where $0<\sigma<\infty$. 
Define $S_n':=X_1+X_2+\ldots+X_n$. 
We are interested in $\P(S_n<0)$ which is equal to $\P(S_n'<k-1)$.
The event $S_n'<(k-1)$ is equivalent to $S_n'-n\mu<(k-1)-n\mu$, which in turn is equivalent to 
$$\dfrac{S_n'-n\mu}{\sqrt{n}\sigma}<\dfrac{(k-1)}{\sqrt{n}\sigma}-\dfrac{\mu}{\sigma}\sqrt{n}.$$
By the central limit theorem (see e.g Grimmett and Stirzaker \cite[\S 5.10]{gsbook}), 
we have $$\dfrac{S_n'-n\mu}{\sqrt{n}\sigma} \xrightarrow[\text{}]{\text{~~D~~}} N(0,1) \mbox{ as }n \to \infty.$$ 
Therefore, for $n$ sufficiently large, 
$$\mathbb{P}(S_n'<k-1)\to \mathbb{P}\left(Y<\dfrac{(k-1)}{\sqrt{n}\sigma}-\dfrac{\mu}{\sigma}\sqrt{n}\right)\mbox{, where }Y \sim N(0,1).$$
But $\dfrac{(k-1)}{\sqrt{n}\sigma}-\dfrac{\mu}{\sigma}\sqrt{n} \to -\infty$ as $n \to \infty$, so $\mathbb{P}(S_n'<(k-1))\to 0$. 
This implies $\mathbb{P}(S_n\geq0)\to1$ as $n\to\infty$. 
As the random walk moves in the positive direction by only taking unit steps, the event $S_n \geq 0$ guarantees that $S_i=0$ for some $i\leq n$.
Therefore $$\ds\sum_{i \geq 0} s^{(k)}_i p^{(k-1)i}q^{i-1} = \frac{1}{q} S^{(k)}(p^{k-1}q)=1.$$ 

Now suppose that $q>1/k$.
Notice that on the closed interval $[0,1]$ the continuous function $f(q):=p^{k-1}q$ takes the minimum value 0 at $q=0$ and $q=1$.
It takes a local maximum uniquely at $q=1/k$. 
Therefore, by the intermediate value theorem, for every $q_1>1/k$ there exists a $q_0<1/k$ such that $$p_0^{k-1}q_0 = p_1^{k-1}q_1,$$
where $p_0:=1-q_0$ and $p_1:=1-q_1$. 
We have already seen that $\frac{1}{q_0} S^{(k)}(p_0^{k-1}q_0) = 1$ for any $q_0<1/k$.
Therefore, for $q_1>\frac{1}{k}$, we have 
\begin{align*}
\frac{1}{q_1} S^{(k)}(p_1^{k-1} q_1) 
&= \frac{q_0}{q_1} \frac{1}{q_0} S^{(k)}(p_1^{k-1} q_1) \\
&= \frac{q_0}{q_1} \frac{1}{q_0} S^{(k)}(p_0^{k-1} q_0) \\
&= \frac{q_0}{q_1} \cdot 1  <1,
\end{align*}
since $q_0<q_1$.
\end{proof}

A more thorough treatment of this random walk can be found in Sen and Saran~\cite{sensaran}.

\begin{lemma}\label{lemma:S_k(z)_firstDerivative}
The derivative
$$\left. \pdz S^{(k)} (p^{k-1}qz)\right|_{z\leftarrow 1}= 
	\begin{cases}
	\dfrac{pq}{1-kq}  & \mbox{ if } q<\tfrac{1}{k} \\[1em]
	\dfrac{(1-q\lambda_{k,p})p^{k-1}q}{(1-q\lambda_{k,p})^{k} - (k-1)p^{k-1}q} & \mbox{ if } q>\tfrac{1}{k} .
	\end{cases}
$$
\end{lemma}

\begin{proof}
By Lemma~\ref{RadiusOfConvergence}, $S^{(k)}(p^{k-1}qz)$ is holomorphic in a region around $z=1$, for $q \neq \dfrac{1}{k}$, 
and so we can evaluate $\left.\pdz S^{(k)} (p^{k-1}qz)\right|_{z\leftarrow 1}$. 
Starting with $S^{(k)}(p^{k-1}qz)=\dfrac{p^{k-1}qz}{(1-S^{(k)}(p^{k-1}qz))^{k-1}}$, we have:
\begin{align*}
    &\pdz S^{(k)} (p^{k-1}qz) \\ & =p^{k-1}qz(1-k)(1-S^{(k)}(p^{k-1}qz))^{-k}\left(-\pdz S^{(k)}(p^{k-1}qz)\right)+p^{k-1}q(1-S^{(k)}(p^{k-1}qz))^{1-k}.
\end{align*}
By rearranging, one now has an expression for the derivative of $S^{(k)}(p^{k-1}qz)$:
\begin{align}
\pdz S^{(k)}(p^{k-1}qz) & = \dfrac{p^{k-1}q(1-S^{(k)}(p^{k-1}qz))}{(1-S^{(k)}(p^{k-1}qz))^{k}-(k-1)p^{k-1}qz}. \label{eq:S_k(z)_firstDerivative}
\end{align}
We wish to evaluate this at $z=1$. Lemma~\ref{probeventualreturn} gives us two cases: 
$$S^{(k)}(p^{k-1}q)= 
	\begin{cases}
	q & \mbox{ if } q<\frac{1}{k} \\
	q\lambda_{k,p} & \mbox{ if } q>\frac{1}{k}.
	\end{cases}
$$
The result follows.
\end{proof}

\begin{lemma}\label{lemma:S_k(z)_secondDerivative}
For $q < \dfrac{1}{k}$,  
$$ \ppdz\left. S^{(k)}(p^{k-1}qz)\right|_{z\leftarrow 1}=\dfrac{pq^2(k-1)(2-qk)}{(1-qk)^3}.$$
\end{lemma}

\begin{proof}
Differentiating equation (\ref{eq:S_k(z)_firstDerivative}), we have:
\begin{align*}
\ppdz S^{(k)}(p^{k-1}qz)
	= & \dfrac{-p^{k-1}q\pdz S^{(k)}(p^{k-1}qz)}{(1-S^{(k)}(p^{k-1}qz))^k-(k-1)p^{k-1}qz}\\
	&  + \dfrac{p^{k-1}q(1-S^{(k)}(p^{k-1}qz))}{((1-S^{(k)}(p^{k-1}qz))^k-(k-1)p^{k-1}qz)^2}(-1) \\
    & \qquad\times\left[k(1-S^{(k)}(p^{k-1}qz))^{k-1}(-\pdz S^{(k)}(p^{k-1}qz))-(k-1)p^{k-1}q\right]\\
\end{align*}
Evaluating this at $z=1$, and using Lemmas~\ref{probeventualreturn} and \ref{lemma:S_k(z)_firstDerivative}, yields the result.
\end{proof}

We are now in a position to prove Theorem~\ref{asymptotics_theorem}.

\begin{proof}\textbf{(Proof of Theorem~\ref{asymptotics_theorem})}
In order to determine the asymptotics of $M^{(k)}_n(p)=[z^n]M^{(k)}(z)$, 
we need to determine the nature of the poles of the function
$$M^{(k)}(z) = \left(k-\dfrac{1}{q}\right)\dfrac{z}{(1-z)^2} 
	+ \dfrac{p}{q} \left(\dfrac{z}{1-z}\right) \left(\dfrac{1}{1-\tfrac{1}{q} S^{(k)}(p^{k-1}qz)}\right).$$
If $\frac{1}{q}S^{(k)}(p^{k-1}qz)=1$ then $z=1$, by Lemma~\ref{technical_lemma_one}, so the only pole is at $z=1$. 
We split our proof into three cases.\\[1em]
\noindent {\textbf{Case $q>\frac{1}{k}$}}:
By Lemma~\ref{probeventualreturn} we have $\frac{1}{q} S^{(k)}(p^{k-1}q)=:\lambda_{k,p} \in (0,1)$. 
Now
\begin{align*}
[z^n]M^{(k)}(z) 
	&= [z^n] \left(\left(k-\dfrac{1}{q}\right)\cdot \dfrac{z}{(1-z)^2} + \dfrac{p}{q}\cdot \dfrac{z}{1-z}\cdot \dfrac{1}{1-\frac1qS^{(k)}(p^{k-1}qz)}\right) \\
	&= \left(k-\dfrac{1}{q}\right)n + [z^n]\dfrac{p}{q}\cdot \dfrac{z}{1-z}\cdot \dfrac{1}{1-\frac1qS^{(k)}(p^{k-1}qz)}.
\end{align*}
As $\frac{1}{q}S^{(k)}(p^{k-1}qz) \neq 1$ at $z=1$, we have that 
$$\dfrac{1}{1-\frac{1}{q}S^{(k)}(p^{k-1}qz)} = \displaystyle\sum_{n \geq 0}c_n(1-z)^n,$$ 
for some constants $c_n$. We can set $z=1$ to determine $c_0=\frac{1}{1-\lambda_{k,p}}$. 
Therefore 
\begin{align*}
\dfrac{p}{q}\cdot \dfrac{z}{1-z}\cdot \dfrac{1}{1-\frac1q S^{(k)}(p^{k-1}qz)}
	&= \dfrac{p}{q}\cdot \left(\dfrac{1}{1-z}-1\right)\cdot \ds\sum_{n \geq 0}c_n(1-z)^n \\
	&= \dfrac{p}{q}\cdot \dfrac{1}{1-z}\cdot \dfrac{1}{1-\lambda_{k,p}}+\ds\sum_{n \geq 0}\alpha_n(1-z)^n,
\end{align*} for some constants $\alpha_n$, $n \geq 0$. 
It follows that 
$$[z^n]\dfrac{p}{q}\cdot \dfrac{z}{1-z}\cdot \dfrac{1}{1-\frac1qS^{(k)}(p^{k-1}qz)} 
	= \dfrac{p}{q}\cdot \dfrac{1}{1-\lambda_{k,p}} +  \mathcal{O} \left(\left(\frac{1}{R^{*}}+\epsilon\right)^n\right),$$ 
for any given $\epsilon > 0$. (See e.g. Wilf~\cite[Theorem 5.5]{wilf}) Here $R^{*}$ is the radius of convergence of $\frac1qS^{(k)}(p^{k-1}qz)$, as derived in Lemma~\ref{RadiusOfConvergence}.
Bringing these results together we have:
$$[z^n]M^{(k)}(z) = \left(k-\dfrac{1}{q}\right)n + \dfrac{p}{q}\cdot \dfrac{1}{1-\lambda_{k,p}} +  \mathcal{O} \left(\left(\dfrac{1}{R^{*}}+\epsilon\right)^n\right),$$
as required.\\[1em]
\noindent \textbf{Case $q<\frac{1}{k}$}:
Let us write 
$$M^{(k)}(z)= \left( k-\frac{1}{q}\right)\dfrac{z}{(1-z)^2} + \dfrac{1}{z-1}H^{(k)}(z),$$
where $H^{(k)}(z):=\frac{pz}{S^{(k)}(p^{k-1}qz)-q}$. 
Consider first the function $H^{(k)}(z)$.
We know from Lemma~\ref{probeventualreturn} that $H^{(k)}(z)$ has a pole at $z=1$.
Noting that the derivative 
$\pdz \left.\left(S^{(k)}(p^{k-1}qz)-q\right)\right|_{z\leftarrow 1}  \neq 0,$ 
by Lemma~\ref{lemma:S_k(z)_firstDerivative}, shows that this is a simple pole. 
Therefore we can write 
$$H^{(k)}(z)=\ds\sum_{n \geq -1}c_n(z-1)^n,$$ and now seek to determine both $c_{-1}$ and $c_0$. 
First notice that 
\begin{align*}
c_{-1}=& \textrm{Res}(H^{(k)}(z):z=1)
	=\dfrac{\left.pz\right|_{z\leftarrow 1}}{\frac{d}{dz}\left.\left(S^{(k)}(p^{k-1}qz)-q\right)\right|_{z\leftarrow 1}}
	=\dfrac{p}{\frac{pq}{1-kq}}
	=\frac{1}{q}-k.
\end{align*}
We can now calculate
$$c_0=\left.\left(H^{(k)}(z)-\frac{c_{-1}}{z-1}\right)\right|_{z \leftarrow 1}.$$ 
This value is 
\begin{align*}
c_0 &= \lim_{z\to 1} \dfrac{pz(z-1)-\left(\frac{1}{q}-k\right)\left(S^{(k)}(p^{k-1}qz)-q\right)}{(z-1)\left(S^{(k)}(p^{k-1}qz)-q\right)}\\
	&= \lim_{z\to 1} \dfrac{p(z-1)+pz-\left(\frac{1}{q}-k\right)\pdz S^{(k)}(p^{k-1}qz)}{(S^{(k)}(p^{k-1}qz)-q)+(z-1)\pdz S^{(k)}(p^{k-1}qz)},
\end{align*}
by applying l'H\^opital's rule. 
As both the numerator and denominator of the last expression are zero when $z=1$, we apply l'H\^opital's rule once again to find
\begin{flalign*}
	&&c_0=&\lim_{z\to 1} \dfrac{2p - \left(\dfrac{1}{q}-k\right) \ppdz S^{(k)}(p^{k-1}qz)}{ \pdz S^{(k)}(p^{k-1}qz)+\pdz S^{(k)}(p^{k-1}qz)+(z-1) \ppdz S^{(k)}(p^{k-1}qz)} &&\\
&&  =&
	\dfrac{2p - (1-kq)\dfrac{pq(k-1)(2-qk)}{(1-qk)^3}}{\frac{2pq}{1-kq}}  && \\
&&	=&
	\frac{2p(1-qk)^2 - pq(k-1)(2-qk)}{2pq(1-kq)} &&\\
&& 	=&\dfrac{2(1+q)-kq(6+q-3kq)}{2q(1-kq)},&&
\end{flalign*}
wherein the second equality is the result of applying Lemma~\ref{lemma:S_k(z)_secondDerivative}.
These two values for $c_0$ and $c_{-1}$ allow us to complete the asymptotic analysis for $M^{(k)}_n(p) = [z^n]M^{(k)}(z)$:

\begin{flalign*}
&& M^{(k)}_n(p)
	&=\left(k-\frac{1}{q}\right)n + [z^n] \dfrac{c_{-1}}{(z-1)^2} + [z^n]\dfrac{c_{0}}{z-1} 
		+ \mathcal{O} \left(\left(\frac{1}{R^{*}}+\epsilon\right)^n\right)  &&\\
&&	&=\left(k-\frac{1}{q}\right)n + \left(\frac{1}{q}-k\right)(n+1) - \dfrac{2(1+q)-kq(6+q-3kq)}{2q(1-kq)} + \mathcal{O} \left(\left(\frac{1}{R^{*}}+\epsilon\right)^n\right) && \\
&&	&=\dfrac{(1-k)(kq-2)}{2(1-kq)} + \mathcal{O} \left(\left(\frac{1}{R^{*}}+\epsilon\right)^n\right),&&
\end{flalign*}
for a given $\epsilon > 0$, as required.  
Once again $R^{*}$ is the radius of convergence of $\frac1qS^{(k)}(p^{k-1}qz)$, as derived in Lemma~\ref{RadiusOfConvergence}.\\[1em]
\noindent \textbf{Case $q=\frac{1}{k}$}:
Rewrite equation (\ref{sgeneratingfunction}) as
$S^{(k)}(z)=z\phi(S^{(k)}(z))$ where $\phi(x):=(1-x)^{-(k-1)}$. 
Functional equations of this form have been considered by Flajolet and Sedgewick~\cite[Section VI.7]{flajolet}. 
The functions $\phi(x)$ and $S^{(k)}(z)$ satisfy the conditions of \cite[Theorem VI.6]{flajolet} and, moreover, the unique positive solution to the characteristic equation $\phi(\tau)-\tau\phi'(\tau)=0$ is $\tau=\frac{1}{k}$.
Writing $$\rho=\dfrac{\tau}{\phi(\tau)}=\dfrac{(k-1)^{k-1}}{k^k}$$ 
	for the radius of convergence of $S^{(k)}(z)$ at 0, we therefore have that, near $\rho$, 
	\begin{align*}
	S^{(k)}(z)&=\dfrac{1}{k}-\sqrt{\dfrac{2(k-1)}{k^3}}\sqrt{1-\dfrac{z}{\rho}} 
		+ \displaystyle\sum_{j \geq 2}(-1)^jd_j\left(1-\dfrac{z}{\rho}\right)^{{j}/{2}},
	\end{align*}
	for some computable constants $d_j$.
Since $\rho$ is $p^{k-1}q$ when $q={1}/{k}$ we have, as $z \to 1$,
\begin{align*}
S^{(k)}(p^{k-1}q z) \sim \dfrac{1}{k} - \sqrt{\dfrac{2(k-1)}{k^3}}\sqrt{1-z}.
\end{align*}
Replacing this into the expression for $M^{(k)}(z)$ in Theorem~\ref{gf:one} we find 
\begin{align*}
M^{(k)}(z) \sim \sqrt{\dfrac{k(k-1)}{2}}(1-z)^{-\frac{3}{2}}.
\end{align*}
Therefore, by \cite[Theorem VI.1]{flajolet}, we have:
\begin{align*}
 [z^n]M^{(k)}(z) &\sim \sqrt{\dfrac{k(k-1)}{2}}\dfrac{\sqrt{n}}{\Gamma(3/2)},
\end{align*}
where $\Gamma({3}/{2})=\int_{0}^{\infty} e^{-t}{t}^{1/2} \, dt = {\sqrt{\pi}}/{2}$.
\end{proof}

\section{Diagonal state matters} \label{sec:four}
A natural aspect of the big-chooser process to consider is that of returning to a diagonal state, a state for which all matchboxes have an equal number of matches.
Knuth~\cite{knuth} made the striking observation that the expected residue in the $k=2$ case was equal to the expected value of $i$, where $(n-i,n-i)$ is the first time the two matchboxes are again equal in size after the process begins. 
Stirzaker gave further context for this observation in his paper~\cite{stirzaker}.
In this section we will consider this expected value for general $k$.

\subsection{Expected order of first diagonal return} 
Let $Y:=Y(k,n,p)$ be the order of the first diagonal return for the general $k$-matchbox process that begins in state $(n,n,\ldots,n)$.
That is, the first return to a diagonal state after leaving the initial configuration $(n,n,\ldots,n)$ is at $(n-Y,n-Y,\ldots,n-Y)$.
Define the generating function
$$R^{(k)}(z):=\displaystyle\sum_{n \geq 1}R_n^{(k)}(p)z^n,$$ where
$R^{(k)}_n(p) := \mathbb{E}(Y(k,n,p))$. The first value is clearly $R^{(k)}_1(p) =1$.

\begin{theorem} \label{thm:expectedreturn}
For $n\geq 2$, the expected order of the first diagonal return is
$$R^{(k)}_n(p)     = 
 n
- \frac{1}{q} \sum_{i=1}^{n-1} (n-i) s^{(k)}_i (p^{k-1}q)^i.
$$
Moreover, its generating function is
$$R^{(k)}(z) = \dfrac{z}{(1-z)^2}\left(1-\dfrac{1}{q}S^{(k)}(p^{k-1}qz)\right).$$
\end{theorem}

\begin{proof}
The probability weights and numbers introduced in Section~\ref{sec:three} allow us to write 
$$\P(Y=i) = s^{(k)}_i (p^{k-1}q)^i/q$$
for all $1\leq i \leq n-1$, while
the  remaining probability
$$\P(Y=n) ~=~  1- \frac{1}{q}\left(s^{(k)}_1 (p^{k-1}q)^1 +  s^{(k)}_2 (p^{k-1}q)^2 + \cdots + s^{(k)}_{n-1}  (p^{k-1}q)^{n-1}\right).$$
From this it follows that
\begin{align*}
R^{(k)}_n(p) 
 = & \left(\sum_{i=1}^{n-1} i\cdot s^{(k)}_i (p^{k-1}q)^i/q\right)\\
&  + n\left( 1- \frac{1}{q}\left(s^{(k)}_1 (p^{k-1}q)^1 +  s^{(k)}_2 (p^{k-1}q)^2 + \cdots + s^{(k)}_{n-1}  (p^{k-1}q)^{n-1}\right)\right)\\
=& n- (n-1)s^{(k)}_1 (p^{k-1}q) \frac{1}{q} - (n-2)  s^{(k)}_2 (p^{k-1}q)^2\frac1q - \ldots - s^{(k)}_{n-1}  (p^{k-1}q)^{n-1}\frac1q.
\end{align*}
Figure~\ref{fig:diagonal:returns} illustrates the experimental value of $R^{(3)}_{100}(p)$. \\

The generating function $R^{(k)}(z)=\sum_{n\geq 1}R^{(k)}_n(p) z^n$ is therefore
\begin{align*}
R^{(k)} (z) = & z + \sum_{n\geq 2} \left(  n- (n-1)s^{(k)}_1 (p^{k-1}q) \frac{1}{q} 
	- \ldots - s^{(k)}_{n-1}  (p^{k-1}q)^{n-1}\frac1q \right) z^n \\
= & z + \sum_{n\geq 2} nz^n\\ 
  & -  \frac{1}{q}\sum_{n\geq 2} \left(  (n-1)z^{n-1}\cdot s^{(k)}_1 (p^{k-1}q)z 
	+ \ldots + z\cdot s^{(k)}_{n-1}  (p^{k-1}q)^{n-1}z^{n-1}\right) \\
= & \sum_{n\geq 1} nz^n -  \frac{1}{q} \left( \sum_{n\geq 1} nz^n \right) S^{(k)} (p^{k-1}qz) \\
= & \dfrac{z}{(1-z)^2} - \frac{1}{q} \dfrac{z}{(1-z)^2} S^{(k)}(p^{k-1}qz)\\
= & \dfrac{z(q-S^{(k)}(p^{k-1}qz))}{q(1-z)^2}.\qedhere
\end{align*}
\end{proof}

\begin{figure}
\begin{center}
\includegraphics[scale=0.75]{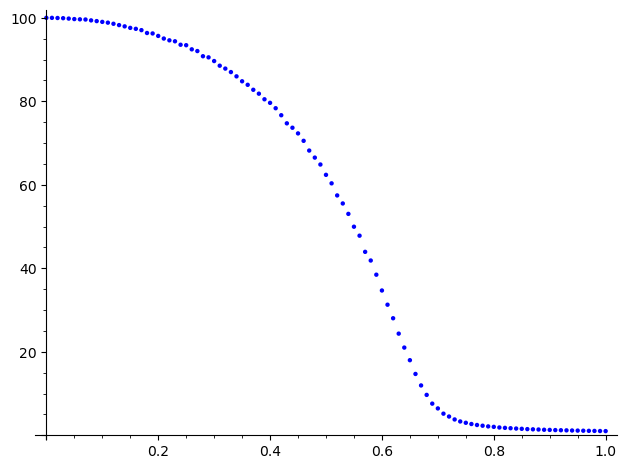}
\caption{The experimental value of the random variable $Y$ where the first return to the diagonal state is at $(n-Y,n-Y,n-Y)$ for $n=100$. 
\label{fig:diagonal:returns}}
\end{center}
\end{figure}

So although $M^{(2)}(z)=R^{(2)}(z)$, we can clearly see that this relationship does not hold for any $k>2$.
It is relatively straightforward to determine the asymptotic behaviour of $R_n(p)$.

\begin{proposition}\label{five:asympt}
    Let $R^{*}=\dfrac{(k-1)^{k-1}}{k^k} \cdot \dfrac{1}{p^{k-1}q}$ and let $\epsilon>0$. Then 
    $$R_n(p)=
    \begin{cases}
    \dfrac{p}{1-kq} + \mathcal{O}\left(\left(\dfrac{1}{R^{*}}+\epsilon\right)^n\right) & \mbox{ if }q<\tfrac{1}{k} \\[1em]
	\sqrt{\dfrac{8(k-1)n}{k\pi}} 
					+
                    \mathcal{O} \left(\dfrac{1}{n^{1/2}}\right)
                         & \mbox{ if }q=\tfrac{1}{k} \\[1em]
    n(1-\lambda_{k,p}) + \dfrac{p^{k-1}(1-q\lambda_{k,p})}{(1-q\lambda_{k,p})^k-(k-1)qp^{k-1}} + \mathcal{O}\left(\left(\dfrac{1}{R^{*}}+\epsilon\right)^n\right) & \mbox{ if }q>\tfrac{1}{k},
    \end{cases}
    $$
    as $n\to \infty$. 
\end{proposition}

\begin{proof}
    As $R^{(k)}(z)=\dfrac{z}{(z-1)^2}\left(1-\dfrac{1}{q}S^{(k)}(p^{k-1}qz)\right)$ we see that, 
	for $q\neq 1/k$, $R^{(k)}(z)$ only has a pole at $z=1$, and that it is of order $2$. 
	For the case $q=1/k$ this is not true and an algebraic singularity at $z = 1$ appears. 
	In any case the (unique) dominant singularity is at $z = 1$.

    Let $1-\dfrac{1}{q}S^{(k)}(p^{k-1}qz)=\displaystyle\sum_{i\geq0}c_i(z-1)^{i}$, and write \\
	$$R^{(k)}(z)=\ds\sum_{j\geq-2}\alpha_j(z-1)^{j} = \left(\dfrac{1}{(z-1)^2}+\dfrac{1}{z-1}\right) 
		\ds\sum_{i\geq0}c_i(z-1)^{i}
		.$$
    We wish to know $\alpha_{-2}$ and $\alpha_{-1}$. 
	We can see that $\alpha_{-2}=c_0$ and $\alpha_{-1}=c_1+c_0$. 
	Now,
    $$ c_0 = \left.\left(1-\dfrac{1}{q}S^{(k)}(p^{k-1}qz)\right)\right|_{z\leftarrow 1}=
    \begin{cases}
    0 & \mbox{ if } q<\tfrac{1}{k} \\
    1-\lambda_{k,p} & \mbox{ if } q>\tfrac{1}{k},
    \end{cases}
    $$
    where $\lambda_{k,p} :=(1-\dfrac{1}{q}S^{(k)}(p^{k-1}q))$ was defined in Lemma~\ref{probeventualreturn}.
    Also, 
	$$c_1
		= \dfrac{d}{dz}\left.\left[1-\dfrac{1}{q}S^{(k)}(p^{k-1}qz)\right]\right|_{z\leftarrow1}=
    	\begin{cases}
    	\dfrac{-p}{1-kq} & \mbox{ if } q<\tfrac{1}{k} \\[1em]
    	\dfrac{-p^{k-1}(1-q\lambda_{k,p})}{(1-q\lambda_{k,p})^k-(k-1)qp^{k-1}} & \mbox{ if } q>\tfrac{1}{k},
    	\end{cases}$$
	by Lemma~\ref{lemma:S_k(z)_firstDerivative}. 
    Finally, let $R^{*}=\dfrac{(k-1)^{k-1}}{k^k} \cdot \dfrac{1}{p^{k-1}q}$ and let $\epsilon>0$. $R^{*}$ is the radius of convergence of $\frac1qS^{(k)}(p^{k-1}qz)$, as derived in Lemma~\ref{RadiusOfConvergence}.  
	Now (e.g. using Wilf~\cite[Theorem 5.5]{wilf}) we have
	$$[z^n]R^{(k)}(z)\to[z^n]\dfrac{\alpha_{-2}}{(z-1)^2}+[z^n]\dfrac{\alpha_{-1}}{(z-1)}+\mathcal{O}\left(\left(\dfrac{1}{R^{*}}+\epsilon\right)^n\right),$$ as $n\to\infty$, from which the first and third results follow.

For the case $q=1/k$, and using $\rho={(k-1)^{k-1}}/{k^k}$ as we did in the proof of Theorem~\ref{asymptotics_theorem}, as $z\to 1$ the generating function
$R^{(k)}(z)={z(1-kS^{(k)}(\rho z))}/{(1-z)^2}$ is asymptotically
\begin{align*}
R^{(k)}(z) & \sim \dfrac{1}{(1-z)^2}\left(1-k\left(\dfrac{1}{k} - \sqrt{\dfrac{2(k-1)}{k^3}}\sqrt{1-z}\right)\right)\\
	& = \sqrt{\dfrac{2(k-1)}{k}}(1-z)^{-\frac{3}{2}}.
\end{align*}
Again, using \cite[Theorem VI.1]{flajolet} as we did in the proof of Theorem~\ref{asymptotics_theorem}, we have
\begin{align*}
[z^n]R^{(k)}(z) &\sim \sqrt{\dfrac{2(k-1)}{k}}\dfrac{2\sqrt{n}}{\sqrt{\pi}},
\end{align*}
as required.
\end{proof}

\subsection{Diagonal state probabilities and manila folders}
\label{sec:five:two}
The quantity $f_n^{(k)}$ is the probability that the system will return to a diagonal state, not necessarily for the first time, after $nk$ choosers have selected matches.
We can give a closed form expression for $f_n^{(k)}$ in terms of a sum, but in order to do so we require some terminology and a lemma.
We will present a closed form expression for $\paths_k(n,i)$ (defined at the end of Section~\ref{sec:two}) by showing the paths that it counts are equinumerous to another seemingly forgotten combinatorial object that was introduced and enumerated by Finucan~\cite{finucan}.

A manila folder is a filing cabinet folder that has a designated number of compartments. 
Each manila folder has $k-1$ compartments so that there are $k$ dividers including the front and back dividers.
Suppose we have $n$ such manila folders.
The folders are placed in a filing cabinet, fronts facing forward (in our diagrams the front is the right-hand side). 
It is permitted to (repeatedly) place a folder inside a compartment of another folder. 

Let $\Fin_k(n,i)$ be the set of possible folder arrangements consisting of $n$ manila folders each having $k-1$ compartments, 
and with precisely $i+1$ of the folder spines visible to a hypothetical observer looking up from below the cabinet.
Let $\manila_k(n,i)=|\Fin_k(n,i)|$.

For example, suppose that we have $n=2$ folders each with $k-1=2$ compartments. 
There are $3$ distinct ways of arranging the folders in a filing cabinet, see Figure~\ref{fig:manila}. 
Two of these arrangements, (b) and (c) in Figure~\ref{fig:manila}, have only one spine visible from below.
Configuration (b) can be represented by the sequence $((\epsilon,\epsilon),\epsilon)$ and configuration (c) by the sequence $(\epsilon,(\epsilon,\epsilon))$, where $\epsilon$ represents an empty compartment. 
So $\manila_3(2,1-1) = \manila_3(2,0)=2$.

The arrangement in (a) has two spines visible from below. 
This arrangement is two empty folders side-by-side and can be represented by the sequence $(\epsilon,\epsilon) (\epsilon,\epsilon)$. 
So $\manila_3(2,2-1)= \manila_3(2,1)=1$.

\finfigone 

Although it is intuitively obvious what $\Manila_k(n,i)$ contains, we find it useful to make a more formal recursive definition.

\begin{definition} \label{recursive:manila}
Fix $k \in \mathbb{N}$. Then, for each $n \in \mathbb{N}$,

\begin{enumerate}
\item[M1.] $\epsilon \in \Manila_k(0,-1)$ (the `no manila folder' element).
\item[M2.] $(M_1,\ldots,M_{k-1}) \in \Manila_k(n,0)$, where $M_j \in \Manila_k(n_j,\cdot)$ for $1 \leq j \leq k-1$, and $n_1+\ldots+n_{k-1}=n-1$.
\item[M3.] $M_1 M_2 \cdots M_{\ell} \in \Manila_k(n,\ell -1)$, where $\ell > 1$, $M_j \in \Manila_k(n_j,0)$ for $1 \leq j \leq \ell$, and $n_1+\ldots+n_{\ell}=n$.
\end{enumerate}
\end{definition}
This recursive definition generates (in a unique way) any configuration of manila folders.
The set of paths in $\Paths_k(n,i)$ has a similar recursive description.

\begin{definition} \label{recursive:paths}
Fix $k \in \mathbb{N}$. Then, for each $n \in \mathbb{N}$,

\begin{enumerate}
\item[P1.] $\epsilon \in \Paths_k(0,-1)$ (the `no path'/'empty path' element).
\item[P2.] $DP_1LP_2L\cdots LP_{k-1}L \in \Paths_k(n,0)$, where $P_j \in \Paths_k(n_j,\cdot)$ for $1 \leq j \leq k-1$, and $n_1+\ldots+n_{k-1}=n-1$.
\item[P3.] $P_1 P_2 \cdots P_{\ell} \in \Paths_k(n,\ell -1)$, where $\ell > 1$, $P_j \in \Paths_k(n_j,0)$ for $1 \leq j \leq \ell$, and $n_1+\ldots+n_{\ell}=n$.
\end{enumerate}
\end{definition}

This recursive definition generates all of the paths in the unit lattice that start at the boundary line $(k-1)y=x$ and take unit south and west steps, staying weakly below that boundary line, until ending at $(0,0)$. To see why P2 works, consider Figure~\ref{andrews:new:diagram}. 
\begin{figure}[!h]
\begin{tikzpicture}
\draw[gray,thin] (0,0) grid (10,5);
\draw[red,very thick] (0,0) -- (10,5);
\draw[black,very thick, dashed] (1,0) -- (10,4.5);
\draw[black,very thick, dashed] (2,0) -- (10,4);
\foreach \x/\y in {10/5,10/4,10/3,10/2,10/1,10/0}{
    \draw[fill=black] (\x,\y) circle (0.5ex);
    }
\node[anchor=east] at (0,0) {$(0,0)$};
\node[anchor=west] at (10,5) {$((k-1)n,n)$};
\node[anchor=west] at (10,4) {$((k-1)n,n-1)$};
\node[anchor=west] at (-1,-2.5) {$B_0:~(k-1)y=x$};
\node[anchor=west] at (1.25,-1.75) {$B_1:~(k-1)y+1=x$};
\node[anchor=west] at (3.5,-1.0) {$B_2:~(k-1)y+2=x$};
\draw[->]        (-0.7,-2.15) .. controls (-0.55, -1.5) and (-0.35,-0.5) .. (0,-0.1);
\draw[->]        (1.55,-1.4) .. controls (1.27,-0.85) .. (1,-0.1);
\draw[->]        (3.5,-1) .. controls (3,-0.8) and (2.2,-0.5) .. (2,-0.1);
\end{tikzpicture}
\caption{Illustration of the decomposition discussion about P2 in the paragraph following it
\label{andrews:new:diagram}} 
\end{figure}
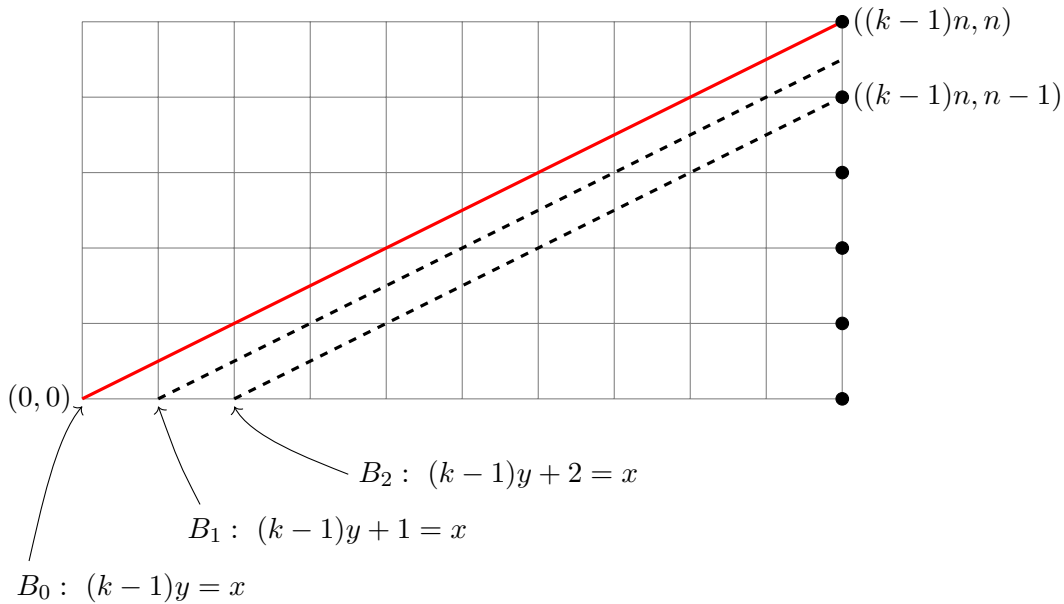
Notice that, for any given path $P$, and any pair of dotted lines $(B_i,B_{i+1})$ ($0 \leq i \leq k-2$), there will necessarily be a unique left step in $P$ that takes $P$ from $B_{i+1}$ to $B_i$ for the first time. This step corresponds to the $(k-1-i)^{th}$ L in the decomposition in P2.

\begin{lemma}\label{pathsfoldersbijection}
The sets $\Fin_k(n,i)$ and $\Paths_k(n,i)$ are equinumerous. 
\end{lemma}

\begin{proof}
This follows by comparing Definitions~\ref{recursive:manila} and~\ref{recursive:paths}.
\end{proof}

Equinumerosity of the sets is sufficient for our purposes but we find it interesting to take a moment and outline a bijective mapping between the two.

Consider the function $\mu: \Fin_k(n,i) \to \Paths_k(n,i)$ defined as follows.
Given $A\in \Fin_k(n,i)$, let $\mu(A)$ be the sequence derived in the following way.
Label the front of each folder $D$ (for {\it{down}}). 
Label each subsequent divider and the back of the folder $L$ (for {\it{left}}). 
Let $\mu(A)$ be the reading of the labels on the folders read in the order that they appear from the front to the back of the cabinet.
Equivalently, if using the bracketed $\epsilon$ representation of $A$, read it from right to left and 
\begin{itemize}
\item for every right-parenthesis that appears append a $D$ to $\mu(A)$,
\item for every left-parenthesis and comma that appears append an $L$ to $\mu(A)$.
\end{itemize}

See Figure~\ref{fig:bij:illustation} for an example of this labelling and the corresponding paths.
For example, for the arrangement $A \in \Fin_3(2,0)$ shown in Figure~\ref{fig:bij:illustation}(c), the label on the front facing folder is $D$. 
Reading the labels we encounter moving from right to left in the diagram we find $\mu(A) = DDLLLL$.

\finfigtwo 

The mapping $\mu$ applied to each of the three types of manila folder configuration in Definition~\ref{recursive:manila} results in the following paths:
\begin{enumerate}
\item[$\mu$1.] $\mu(\epsilon)=\epsilon$, the empty path.
\item[$\mu$2.] $\mu(M_1 M_2 \cdots M_{\ell}) = P_1 P_2 \cdots P_{\ell}$ where $\ell\geq 2$ and each $M_j \in \Fin_k(\cdot,0)$. 
		The resulting path returns to the boundary line precisely $\ell-1$ times between endpoints.
\item[$\mu$3.] $\mu((M_1,\ldots,M_{k-1})) = D P_1 L P_2 L \ldots L P_{k-1} L$ where each $M_j \in \Fin_k(\cdot,\cdot)$. 
		The resulting path does not touch the boundary line between endpoints.
\end{enumerate}

We now present a nice formula for $\paths_k(n,i)$, the number of paths from $((k-1)n,n) \to (0,0)$ that revisit the boundary line $(k-1)y=x$ precisely $i$ times en route.
\begin{proposition}\label{prop:paths}
$\paths_k(n,i)=\frac{i+1}{n}\binom{kn-i-2}{n-i-1}$    
\end{proposition}

\begin{proof}
Finucan showed~\cite[\S 1.3]{finucan} that the number of arrangements of $n$ manila folders that each contain $k$ compartments and for which $h$ spines are hidden (equivalent to $n-h$ spines being visible) is
$\frac{n-h}{n} \binom{nk+h-1}{h}$. 
He\footnote{Finucan actually proved a more general result and this formula is a special case of his equation (4). See Finucan~\cite[Eqn. 4 and \S 4.2]{finucan}.}
called this number $F(k,n,h)$.
As our $\manila_k(n,i)$ counts the number of arrangements of $n$ folders, each with $(k-1)$ compartments, and with $(i+1)$ visible spines, we have:
\begin{align*}
    \manila_k(n,i)
        &= F(k-1,n,n-i-1) \\
        &= \dfrac{n-(n-i-1)}{n}\binom{n(k-1)+n-i-2}{n-i-1} \\
        &= \frac{i+1}{n}\binom{kn-i-2}{n-i-1} \\
        &= \paths_k(n,i),
\end{align*}
by Lemma~\ref{pathsfoldersbijection}.
\end{proof}

This allows us to give an exact expression for $f^{(k)}_n$.
\begin{proposition}\label{fnk:enum}
For $n\geq 1$ we have:
\begin{align*}
f^{(k)}_n=p^{(k-1)n}\displaystyle\sum_{i=0}^{n-1}\frac{n-i}{n}\binom{(k-1)n+i-1}{i}q^{i}.
\end{align*}
\end{proposition}

\begin{proof}
We can interpret $f^{(k)}_n$ using (probability) weighted lattice paths from $\Paths_k(n)$.
By conditioning on the number of returns to the boundary line it follows that
\begin{align*}
f^{(k)}_n
    &= p^{(k-1)n}\displaystyle\sum_{i=0}^{n-1} \paths_k(n,i)q^{n-1-i} \\
    &= p^{(k-1)n}\displaystyle\sum_{i=0}^{n-1} \dfrac{i+1}{n} \binom{n(k-1)+n-i-2}{n-(i+1)} q^{n-1-i} \\	
    &= p^{(k-1)n}\displaystyle\sum_{j=0}^{n-1} \dfrac{n-1-j+1}{n} \binom{n(k-1)+n-(n-1-j)-2}{n-(n-1-j+1))} q^{j} \\	
    &= p^{(k-1)n}\displaystyle\sum_{j=0}^{n-1} \dfrac{n-j}{n} \binom{n(k-1) +j -1}{j} q^{j},
\end{align*}
wherein the second equality uses Proposition~\ref{prop:paths}
\end{proof}

\section*{Acknowledgements}
The authors thank David Stirzaker for sharing a copy of \cite{stirzaker} with them.
They also thank the anonymous referee for valuable suggestions that have improved the paper.

\end{document}